\documentclass[11 pt]{article}  % Comment this line out if you need a4paper

\usepackage{multirow}
\usepackage{amsthm}
\usepackage[section]{placeins}
\usepackage[utf8]{inputenc}
\usepackage{amssymb}
\usepackage{amsmath}
\usepackage{bm}
\usepackage{hyperref}
\usepackage{tikz}
\usepackage{tkz-euclide}
\usepackage{chngcntr}
\usepackage{verbatim}
\usepackage{mathtools}
\usepackage{esvect}
\usepackage{subfiles}
\usepackage{color}
\usepackage{xspace}
\usepackage{xparse}
\usepackage{subfigure}
\usepackage{bbm}

\newtheorem{theorem}{Theorem}
\newtheorem{lemma}[theorem]{Lemma}

\newtheorem{definition}{Definition}
\def\cu#1{{\color{black}#1}}
\usepackage{fullpage}
\usepackage{multirow}
\usepackage{amsthm}
\usepackage[super,sort&compress,numbers]{natbib} 
\usepackage[section]{placeins}

\usepackage[utf8]{inputenc}
\usepackage{amssymb}
\usepackage{amsmath}
\usepackage{bm}
\usepackage{hyperref}
\usepackage{tikz}
\usepackage{tkz-euclide}
\usepackage{chngcntr}
\usepackage{verbatim}
\usepackage{mathtools}
\usepackage{esvect}
\usepackage{subfiles}
\usepackage{color}
\usepackage{xspace}
\usepackage{xparse}
\usepackage{subfigure}
\usepackage{bbm}
\usepackage{enumerate}

\def\cu#1{{\color{black}#1}} % Cesar's notes
\title{\textbf{Graph-Theoretic Analysis of Belief System Dynamics under Logic Constraints}}
\usepackage{authblk}

\author[1,2]{Angelia Nedi\'{c}}
\author[3]{Alex Olshevsky}
\author[4]{C\'{e}sar A. Uribe\thanks{Corresponding Author.}}
\affil[1]{ECEE Department, Arizona State University, Tempe AZ.}
\affil[2]{Moscow Institute of Physics and Technology (MIPT), Dolgoprudny, Russia}
\affil[3]{ECE Department and Division of Systems Engineering, Boston University, Boston MA.}
\affil[4]{Laboratory for Information and Decision Systems (LIDS), and the Institute for Data, Systems, and Society (IDSS), Massachusetts Institute of Technology, Cambridge MA. }

\date{}

\begin{document}
	\maketitle
	\begin{abstract}
		Opinion formation cannot be modeled solely as an ideological deduction from a set of principles; rather, repeated social interactions and logic constraints among statements are consequential in the construct of belief systems. We address three basic questions in the analysis of social opinion dynamics: (i) Will a belief system converge? (ii) How long does it take to converge? (iii) Where does it converge? We provide graph-theoretic answers to these questions for a model of opinion dynamics of a belief system with logic constraints. Our results make plain the implicit dependence of the convergence properties of a belief system on the underlying social network and on the set of logic constraints that relate beliefs on different statements. Moreover, we provide an explicit analysis of a variety of commonly used large-scale network models.
	\end{abstract}
	
	%%%%%%%%%%%%%%%%%%%%%%%%%%%%%%%%%%%%%%%%%%%%%%%%%%%%%%
	%%%%%%%%%%%%%%%%%%%%%%%%%%%%%%%%%%%%%%%%
	
\section*{Introduction}

	The modeling of opinion dynamics spans several decades of interdisciplinary research\cite{con64,fel88,ace11,jac10,heg05,mir12,fri15,car59,fri11}. Belief systems are typically modeled as a process where agents continuously update their opinions on a set of truth statements via repeated interactions, and opinions are exchanged following some social structure\cite{deg74,abe64}. New opinions are formed by aggregating operations weighted by the relative importance assigned by an individual to others. This simple characterization has provided tools for analyzing the long-term behavior of belief systems using systems theory. Nevertheless, without significant modification, this framework has been shown insufficient to explain the existence of shared beliefs in a population\cite{fri16}. 
	
	Recently proposed generalizations of opinion dynamics models integrate functional interdependencies among issues that coherently bound ideas and attitudes\cite{par16}. The existence of \textit{logic constraints} in a belief system provides a successful model for the evolution of opinions in both large-scale populations and small groups\cite{fri16}. \textit{Logic constraints} build upon the natural idea that believing a specific statement is true may depend on the belief that some other related statements are true as well. 
	
	Understanding the role of the networks involved in the structural features of a belief system is of critical importance and can have direct implications for better decision-making and policy design\cite{but16,amb04,for05,van17,fri16}. Even though sophisticated algebraic tools\cite{par16,pros2017} exists for the analysis of opinion dynamics, they can be unpractical or intractable for large-scale complex networks.
	
	In this paper, we study how the structural properties of the social network of agents and the set of logic constraints influence the dynamics of a belief system from a \textit{graph-theoretic} point of view. We describe the combinatorial features which influence the convergence of beliefs, the expected convergence time and the stationary value of the belief system. Informally, we answer the following three questions with graph-theoretic conditions that are easily accessible for a number of commonly used topologies in large-scale complex networks:
	\begin{enumerate}\label{questions}
		\item When does a belief system converge? \label{q1}
		\item How long does it take for a belief system to converge? \label{q2}
		\item Where does a belief system converge? \label{q3}
	\end{enumerate}

\section*{Results}

\subsection*{Belief System with Logic Constraints}

Friedkin et al.\cite{fri16,par16} describe a belief system with logic constraints as a group of $n$ agents that periodically exchange and update their opinions about a set of $m$ different truth statements with logical dependencies among them. After each social interaction, the agents use shared opinions, as well as underlying logical dependencies among them, to update their beliefs. 
	
	The agents exchange their opinions by interacting over a social network captured by a graph $\mathcal{G}= (V,E)$, where $V$ is the set of agents, and $E$ is a set of edges. A directed edge towards an agent indicates that it receives the opinion of another agent, i.e., the directed flow of information. Analogously, the logical dependencies among the truth statements are modeled by a graph 
	$\mathcal{T} = (W,D)$, where an edge between two statements exists if the belief in one statement affects belief in the other.  
	
	The generalized dynamics of a belief system are defined as follows. First, every agent aggregates its opinions on every truth statement according to the imposed logic constraints (i.e., modifying the opinions to consider the dependencies on the other truth statements). 
	Second, the agents share their opinions over a social network,
	where the opinions are aggregated again to take into account those coming from the neighboring agents (i.e., social interactions). 
	Finally, a new opinion is formed as a combination of the most recent aggregation and the initial opinion, modeling the adversity to deviate from the initial beliefs or stubbornness. 
	
The aggregation steps consist of weighted (convex) combinations of the available values, where the weights represent the relative influence. This model is described in the following equations~\eqref{individual_model} for an arbitrary agent $i\in V$ and an arbitrary statement $u\in W$:
	\begin{subequations}\label{individual_model}
		\begin{align}
		\hat x^i_k(u) & = \sum_{v=1}^{m} C_{uv} x^i_k(v) &\text{(Aggregation by logic constraints)}\\
		\bar x^i_k(u) & = \sum_{j=1}^{n}A_{ij} \hat x^j_k(u)&\text{(Aggregation by social network)}\\
		x^i_{k+1}(u) & = \lambda^i \bar x^i_k(u) + (1-\lambda^i) x^i_0(u)&\text{(Influence of initial beliefs)} \label{individual_model3}
		\end{align}
	\end{subequations}
	where ${{0 \leq x^i_k(u) \leq 1 }}$ represents the opinion of an agent $i$ at time $k$ on a certain 
	statement $u$, while $\hat x^i_k(u)$ and $\bar x^i_k(u)$ are the intermediate aggregation steps. The opinion of an agent on a specific statement being true or false is modeled by a scalar value between zero and one. A value of zero indicates that the given agent strongly believes a specific statement is false, whereas a value of one indicates that the agent believes the statement is true. Similarly, a value of $0.5$ indicates the maximal uncertainty about a statement.  
	
	The intermediate aggregated opinion $\hat x^i_k(u)$ of agent $i$ on statement $u$ 
	is formed by using the opinions of the same agent about the other statements $v$.
	The parameters $0 \leq C_{uv}\leq 1$ are compliant with the graph $\mathcal{T}$ that models the logic constraints in the sense that $C_{uv}$ is nonzero if the statement $u$ depends on statement $v$, and 
	otherwise $C_{uv}=0$. These parameters represent the strength of the logic constraints, i.e., 
	the influence that an opinion on a statement has on the opinion on other statements.
	Subsequently, the intermediate aggregated opinion $\bar x^i_k(u)$ of agent $i$ on statement $u$ is
	formed by combining all the intermediate opinions $\bar x^i_k(u)$ of neighboring agents $j$. In this update,
	the parameters $0 \leq A_{ij} \leq 1$ represent the weights that an agent $i$ assigns to the information coming from its neighbor $j$, for example $A_{13}$ is how agent $1$ weights the opinions shared by agent $3$. These parameters are compliant with the network $\mathcal{G}$ in the sense that 
	if there is an incoming edge to agent $i$ from agent $j$ in the graph, 
	then the corresponding weight $A_{ij}$ is nonzero.
	
	Equation~\eqref{individual_model3} indicates that, at time $k+1$,
	the new opinion $x_{k+1}^i(u)$ of agent $i$ on statement $u$ is obtained as a weighted combination of its intermediate aggregated opinion $\bar x^i_k(u)$ at time $k$ and its initial opinion $x^i_0(u)$ on statement $u$.
	The parameter $0\leq \lambda^i \leq 1$ that agent $i$ uses models its stubbornness. 
	If $\lambda^i <1$ we say an agent is \textit{stubborn}, where $\lambda^i=0$ indicates that the agent $i$ is \textit{maximally closed} to the influence of others. If $\lambda^i=1$, agent $i$ is said to be \textit{maximally open} to the influence of others, and \textit{oblivious} if additionally, it is not influenced by stubborn agents.

	We can group the parameters $\{A_{ij}\}$ into 
	an $n$-by-$n$ matrix $A$, known as the \textit{social influence structure}, 
	and the parameters $\{C_{uv}\}$ into an 
	$m$-by-$m$ matrix $C$, known as the \textit{multi-issues dependent structure}\cite{par16}. We assume these matrices are nonnegative. 
	Furthermore, the weights $A_{ij}$ assigned by an agent $i$ to its neighbors $j$
	sum up to one, i.e., the sum of the entries in each row of the matrix $A$ is $1$; likewise, the sum of the entries in each row of the matrix $C$ is $1$. 
	Thus, the matrices $A$ and $C$ are row-stochastic.

	Figure~\ref{networks1} illustrates a belief system with $4$ agents and $3$ truth statements,
	moreover, it gives examples for the choice of the matrices $A$ and $C$. Figure~\ref{networks1}(c) shows the belief system generated by the network of agents in 
	Fig.~\ref{networks1}(a) and the set of logic constraints in Fig.~\ref{networks1}(b). 
	This new graph depicted in Fig.~\ref{networks1}(c) is much larger than the network of agents or the network of statements taken separately; effectively; it has $2nm$ nodes. 
	The belief of each agent on each truth statement is a separate node; also, 
	the initial beliefs are separate nodes. 
	
	The model of this larger graph of the belief system can be compactly restated as
	\begin{align}\label{consensus}
	x_{k+1} = P x_k,
	\end{align}
	where $x_k \in [0,1]^{2nm}$ is a state that stacks the current 
	beliefs of all agents on all topics alongside with the initial beliefs, i.e., 
	\begin{align*}
	x_k & = \bigg[ \underbrace{x^1_k(1),\hdots,x^1_k(m)}_{\text{Beliefs of Agent $1$}},\hdots,\underbrace{x^n_k(1),\hdots,x^n_k(m)}_{\text{Beliefs of Agent $n$}}, \underbrace{x^1_0(1),\hdots,x^1_0(m)}_{\text{Initial Beliefs of Agent $1$}},\hdots,\underbrace{x^n_0(1),\hdots,x^n_0(m)}_{\text{Initial Beliefs of Agent $n$}}\bigg] '
	\end{align*}
	and
	\begin{align*}
	P & = \left[ 
	\begin{array}{c | c}
	(\Lambda A)\otimes C & (\boldsymbol{I}_n-\Lambda)\otimes \boldsymbol{I}_{m}\\
	\hline
	\boldsymbol{0}_{nm}&\boldsymbol{I}_{nm}
	\end{array}
	\right] ,
	\end{align*}
	where $\boldsymbol{0}_{nm}$ is a zero matrix of size $n \times m$, $\boldsymbol{I}_{nm}$ is an identity matrix of size $n \times m$, $\otimes$~indicates the Kronecker product 
	(see Supplementary Definition~\ref{def:kronecker}), $\Lambda$ is a diagonal matrix with the $i$-th diagonal entry being $\lambda^i$, and $x'$ denotes the transpose of a vector or matrix $x$. This allows for the definition of the belief system graph $\mathcal{P}$, 
	which is compliant with the matrix $P$, where an edge from $\ell$ to $r$ exists if $P_{r\ell} > 0$. 

	%See Supplementary equation \eqref{supp:p} for an example of a matrix $P$ for the belief system in Fig.~\ref{networks1}(c) assuming that $\lambda^i=0.5$ for all agents.
	
	Figure~\ref{fig:ex1} shows an example where a network of $5$ agents forms a cycle graph, 
	given in Fig.~\ref{fig:ex1}(a), a set of $4$ logic constraints forms a directed path, given in Fig.~\ref{fig:ex1}(b), and $\lambda^i = 1$ for all $i$.
	The belief system graph is shown in Fig.~\ref{fig:ex1}(c). 
	Figure~\ref{fig:ex1}(d) shows the dynamics of the belief vector as the number of social interactions increases. 
	The opinion on all $4$ topics converges to a single value for all agents. 
	Figure~\ref{fig:ex1}(e) shows the dynamics of the belief vector when no logic constraints are considered. In this case, the agents reach some agreement on the final value, but this consensual value is different for each of the statements. See Supplementary Fig.~\ref{learning_network} for an additional example of the influence of the logic constraints on the resulting belief system and Supplementary Fig.~\ref{fig:supp4} 
	for a variation of the example discussed in Fig.~\ref{fig:ex1} when the network of agents is a complete graph.

\subsection*{When does a Belief System Converge?}

	The convergence of the belief system can be stated as a question of the existence of a 
	limit of the beliefs, as the social interactions continue with time. 
	That is, whether there exists a vector of opinions $x_\infty$ such that
	\begin{align*}
	\lim\limits_{k \to \infty} x_k = \lim\limits_{k \to \infty} P^kx_0 & = x_\infty ,
	\end{align*}
	for any initial value $x_0$. 
	
	Friedkin et al.\cite{fri16,par16} showed that a belief system with logic constraints will converge to equilibrium if and only if either ${{\lim_{k \to \infty} (\Lambda A)^k =0 }}$, or ${{\lim_{k \to \infty} (\Lambda A)^k \neq 0}}$ and ${{\lim_{k \to \infty} C^k }}$ exists. Moreover, if we represent the matrices $A$ and $\Lambda$ with a block structure as
	\begin{align*}
	A & = \left[\begin{array}{ c c}
	A^{11} & A^{12}\\
	0    & A^{22}
	\end{array}\right] \qquad \text{and} \qquad \Lambda = \left[\begin{array}{ c c}
	\Lambda^{11} & 0\\
	0    & I
	\end{array}\right],
	\end{align*}
	where $A^{22}$ is the subgraph of oblivious agents, then the belief system is convergent if and only if ${{\lim_{k \to \infty} C^k }}$ and ${{\lim_{k \to \infty} (A^{22})^k }}$ exists. We next consider how these conditions may be interpreted in terms of the topology of the network of agents and the set of logic constraints.

	The belief system in equation \eqref{consensus} converges to equilibrium if and only if every closed strongly connected component of the graph $\mathcal{P}$ is aperiodic~\cite{jac10,pro17}. Recall that a strongly connected component is closed if it has no incoming links from other agents; otherwise, it is called open, see Fig. \ref{fig:cond}. In general, the set of strongly connected components can be computed efficiently for large-complex networks\cite{tar72}.

	The matrix $P$ has two diagonal blocks, one corresponding to the initial beliefs and one involving the product $\Lambda A\otimes C$. The initial belief nodes are aperiodic closed strongly connected components, 
	each consisting of a single node. Therefore, the diagonal block in $P$ corresponding to the initial beliefs induces an aperiodic graph. Moreover, strongly connected components with stubborn agents do not affect the convergence of the belief system. Thus, one can focus on the closed strongly connected components of the graph induced by $A^{22} \otimes C$.
	
	Lemma~\ref{lemma:kron} characterizes the strongly connected components of the product of two graphs. Particularly, it shows that $A^{22}\otimes C$ can be written in a block upper triangular form, where each of the blocks in the diagonal is the product of one strongly connected component from the graph induced by $A^{22}$ and one from $\mathcal{T}$. 
	
	\begin{lemma}\label{lemma:kron}
		Given two graphs $\mathcal{G}_1$ and $\mathcal{G}_2$, every strongly connected component of the Kronecker product graph $\mathcal{G}_1 \otimes \mathcal{G}_2$ 
		is the result of the Kronecker product of a strongly connected component of $\mathcal{G}_1$ and a strongly connected component of $\mathcal{G}_2$. 
	\end{lemma}

\begin{proof} Let $A_1$ and $A_2$ denote the adjacency matrices for the graphs $\mathcal{G}_1$ and $\mathcal{G}_2$, respectively. {We can construct a condensation of the graph $\mathcal{G}$ by contracting every strongly connected component to a single vertex, resulting in a directed acyclic graph. Thus, a topological ordering is possible (see Cormen et al.~\cite{Cormen2009} Section $22.4$) and} there always exists two permutation matrices $P_1$ and $P_2$ such that we can rearrange the matrices $A_1$ and $A_2$ into a block upper triangular form where each of the blocks is a strongly connected component, that is
	\begin{align*}
	P_1'A_1P_1 = \left[\begin{array}{ c c c c}
	A_1^1 & * & * & *\\
	0 & A_1^2 &  * & * \\
	0 & 0 & \ddots & * \\
	0 & 0 & \hdots & A_{1}^{n_1}\\
	\end{array}\right] \qquad\ \text{and}\qquad\
	P_2'A_2P_2 = \left[\begin{array}{ c c c c}
	A_2^1 & * & * & *\\
	0 & A_2^2 &  * & * \\
	0 & 0 & \ddots & * \\
	0 & 0 & \hdots & A_{2}^{n_2}\\
	\end{array}\right].
	\end{align*}

	Moreover, define $P = P_1 \otimes P_2$ and by the properties of the Kronecker product, cf., Definition~\ref{def:kronecker}, it follows that
	\begin{align*}
	(P_1'A_1P_1)\otimes (P_2'A_2P_2)& = P'(A_1\otimes A_2) P,
	\end{align*}
	where $P$ is also a permutation matrix and
	\begin{align*}
	P'(A_1\otimes A_2) P = \left[\begin{array}{ c c c }
	A_1^1 \otimes A_2 &  *         & * \\
	0                 &  \ddots     &  * \\
	0                 &  \cdots     & A_{1}^{n_1}  \otimes A_2\\
	\end{array}\right] .
	\end{align*}
	
	Finally, by property $2$ in Definition  \ref{def:kronecker} there exists a permutation matrix $Q$ such that
	\begin{align*}
	Q' (P'(A_1\otimes A_2) P ) Q & = \left[\begin{array}{ c c c }
	A_2 \otimes A_1^1 &  *         & * \\
	0                 &  \ddots     &  * \\
	0                 &  \cdots     & A_2  \otimes A_{1}^{n_1}\\
	\end{array}\right] \\
	&= \left[\begin{array}{ c c c c c c c}
	A_2^1 \otimes A_1^1 &  *         & * & * & * & * & * \\
	0                 &  \ddots     &  * &  * & * & * & *\\
	0                 &  \cdots     & A_{2}^{n_2}  \otimes A_1^1 & * & * & * & *\\
	0                 &  \cdots     & 0                         & \ddots & * & *\\
	0                 &  \cdots     & \cdots                        & 0 & A_{2}^1  \otimes A_{1}^{n_1}  & * & *\\
	0                 &  \cdots     & \cdots                        & \cdots & 0 & \ddots  & *\\
	0                 &  \cdots     & \cdots                        & \cdots & \cdots & 0 &  A_{2}^{n_2}   \otimes A_{1}^{n_1}\\
	\end{array}\right]  .
	\end{align*}
	
	Therefore, every block in the upper triangular block diagonal form of the product of two adjacency matrices is the product of two strongly connected components, one from each graph.
\end{proof}

With Lemma~\ref{lemma:kron} at hand, the following result provides a graph-theoretic condition for the convergence of a belief system, complementing the algebraic criteria derived by Friedkin et al.\cite{fri16,par16}.
	
	\begin{theorem}
		The process~\eqref{consensus} converges to equilibrium if an only if every closed strongly connected component of the graph 
		$\mathcal{T}$ is aperiodic and every closed strongly connected component of the graph $\mathcal{G}$ composed by oblivious agents only is aperiodic. 
	\end{theorem}

\begin{proof}
		McAndrew~\cite{McAndrew1963} showed that the period of a product graph is the lowest common multiple of the periods of the two factor graphs (see Supplementary Definition~\ref{def:prod_graph} and Supplementary Theorem~\ref{thm:kron_directed}). If the factor graphs are not coprime, the resulting product graph is a disconnected set of components. Nevertheless, each of the resulting components will have the same period as defined above. Therefore, for a product graph to be aperiodic, we require the factors to be aperiodic as well. Thus, the desired result follows from Lemma~\ref{lemma:kron}.
\end{proof}

	In Fig.~\ref{networks1}, the network of agents has a single closed strongly connected component which consists of the node $4$. The network of truth statements also has a single closed strongly connected component, consisting of the node $3$. Thus, the belief system will converge to a set of final beliefs. 
	In Fig.~\ref{fig:ex1}, the belief system has one closed strongly connected component shown in green with the topology of a cycle graph. 
	This strongly connected component corresponds to the product of the cycle graph and the green node of the logic constraints. The cycle graph is aperiodic if and only if the number of nodes is odd. Thus, if the cycle network of agents has an even number of nodes, the belief system will not converge.

\subsection*{How long does a Belief System take to Converge?}

	We seek to characterize the time required by the process in equation~\eqref{consensus} to be arbitrarily close to its limiting value in terms of properties of the graphs $\mathcal{G}$ and $\mathcal{T}$, 
	such as the number of agents and truth statements, and the topology of the graphs. 
	
	We provide an estimate on the number of iterations required for the beliefs to $\epsilon$ apart from their final value (assuming they converge). This estimate is expressed in terms of the total variation distance, denoted by $\|\cdot\|_{TV}$
	(for its definition see the section on Methods). 
	For this we define the convergence time as follows:
	\begin{align*}
	t_{\text{mix}}(\epsilon) & = \min_{k\ge0} \left\lbrace k :  \cu{\max_{x_0 \in S_1(2nm)} }\|x_k - x_\infty\|_{TV} \leq \epsilon \right\rbrace, 
	\end{align*}
	where \cu{$\varepsilon = 1/4$ is a common choice, $S_1(n)  = \{ x \in [0,1]^n  \mid \sum_{i=1}^n x_i =1 \}$} and $x_k$ follows equation \eqref{consensus}. Informally, $t_{\text{mix}}(\epsilon)$ is the minimum number of social interactions required for the belief system to be arbitrarily close to its final value \cu{for the worst case} initial disagreement.
	
	The dynamics of the belief system in equation~\eqref{consensus} are closely related to the dynamics of a Markov chain with a transition matrix $P$\cite{Olshevsky2013}, specifically, the ergodic properties of a random walk over on the graph $\mathcal{P}$. Particularly, consider a random walk on the state space $\{1,\hdots,2nm\}$ which, at time $k$ jumps to a random neighbor of its current state. The relation between a random walk on a graph and the convergence properties of systems of the form of the belief system in~\eqref{consensus} has been previously explored in Olshevsky and Tsitsiklis~\cite{Olshevsky2013}. In both cases, we are interested in the convergence properties of $P^k$ as $k$ goes to infinity. If there is a limiting distribution for a Markov chain with transition probability $P$, then the belief system converges. Moreover, bounds on the convergence time based on the mixing properties of this Markov chain provide rates of convergence for the belief system.
	
	Next, we will show that the convergence time of a belief system is proportional to the maximum time required for a random walk, with transition probability matrix $P$, to get \textit{absorbed} into a closed strongly connected component in addition to the time needed for such component to \textit{mix} sufficiently. 
	
	\begin{lemma}\label{theorem:time}
		Let $\mathcal{P}$ be a graph with at least one closed strongly connected component, and assume all its closed strongly connected components are aperiodic. Also, let $L$ be the maximum expected coupling time of a random walk in a closed strongly connected component of $\mathcal{P}$. Moreover, let $H$ be maximum expected time for a random walk, starting at an arbitrary node, to get absorbed into a closed strongly connected component. Then, for $k \geq 4(L+H) \log( 1 /\epsilon)$, it holds for the belief system described in equation \eqref{consensus} that  $\|x_k - x_\infty\|_{TV} \leq \epsilon$.
	\end{lemma}

\begin{proof}
	We use the coupling method to bound the convergence time of the belief system \cite{lin02}. Initially, we show that all opinions $x^i_k$, such that $i$ lies in a closed strongly connected component, will converge to some stationary point. Thus, in what follows we will find the required time to reach some $\epsilon$-consensus via coupling arguments, which in turn will provide the required time for a belief system to be $\epsilon$ close to its stationary distribution.
	
	Let $i$ be a node belonging to a closed strongly connected component $S$ and let $P_S$ be the matrix obtained by looking at the minor of $P$ corresponding to entries in $S$. If $S$ is closed then $P_S$ is row-stochastic, and Perron-Frobenius theory tells us there exists some vector $\pi_S$ such that
	\begin{align*}
	\pi_S' P_S = \pi_S'.
	\end{align*}
	
	Now, define two independent random walks $X = (X_k)_0^{\infty}$ and $Y = (Y_n)_0^{\infty}$ with the same transition matrix $P_S$. $X$ starts from the distribution $\pi_S$, and $Y$ from some other arbitrary stochastic vector $v$. Moreover, \textit{couple} the processes $Y$ and $X$ by defining a new process $W$ such that
	\begin{align*}
	W_k = \begin{cases}
	Y_k, & \text{if}\ k <K, \\
	X_k, & \text{if}\ k \geq K ,
	\end{cases}
	\end{align*}
	where $K = \min\left\lbrace k \geq 0 : Y_k = X_k\right\rbrace $ is  called the \textit{coupling time}. Each random walk moves according to $P_S$, so if we correlate them by moving them together after they intersect, we have not changed the fact that, individually, they move according to $P_S$. With this construction of the coupling\cite[Theorem~5.2]{lev09}, we have that
	\begin{align*}
	\|v'P^k_S - \pi_S\|_{TV} & \leq  \max_{v} \mathbb{P}\left\lbrace K > k\right\rbrace,
	\end{align*}
	and by the Markov inequality
	\begin{align*}
	\|v'P^k_S - \pi_S\|_{TV} & \leq  \frac{\max_{v} \mathbb{E}[K]}{k}.
	\end{align*}
	Therefore, to be at a distance of at most $1/4$ we require $k = 4 \max_v \mathbb{E}[K]$. We say the mixing time of the random walk is $4 L$ where we have that $L = \max_{v} \mathbb{E}[K]$ is the maximum expected time it takes for the random walks $X$ and $Y$ in $S$ to intersect. Then, it follows that in order to be $\epsilon$ close to the stationary distribution we require at least $k \geq 4L \log (1/\epsilon)$ steps\cite[Eq.~$4.36$]{lev09}, for any $v$. Therefore, we have shown that $x^i_k$ for $i$ in a closed strongly connected component $S$ converges to $\pi_S' x^S_0$ at a geometric rate. Here $x^S_0$ stacks those $x^i_0$ that belong to $S$.
	
	Now, consider the case where $i$ belongs to an open strongly connected component. Let $M$ be the set of states in such connected component. Stacking up $x^i_k$ over $i$ in $M$ into the vector $x^M_k$, observe that 
	\begin{align}\label{eq:nodes_open}
	x^M_{k+1} = Z x^M_k + Ry_k,
	\end{align}
	where $Z$ is strongly connected and substochastic, meaning some rows add up to less than $1$. The entries of $y_k$ come from nodes in other strongly connected components and the matrix $R$ represents how they influence the nodes in $M$. 
	
	Initially, assume that $y_k$ converges and call its limit $y_{\infty}$. Now, consider a random walk that moves around $M$ according to $Z$; the moment it steps out of $M$ into another strongly connected component we say it is absorbed by it since it can not return to $M$. 
	
	Let $q^i_k$ be the probability the walk is at state $i$ in $M$ at time $k$. Then
	\begin{align*}
	q'_{k+1} = {q_k}' Z ,
	\end{align*}
	and let $H_i$ be the expected time to get absorbed into any other strongly connected component, the set of nodes in $M$ is connected to, starting from node $i$ and let 
	\begin{align*}
	H^1 = \max_{i \in M} H_i.
	\end{align*}
	
	{If the absorbing strongly connected component is closed, then $H = H^1$. On the other hand, the absorbing strongly connected component will have some other absorbing time $H^2$, i.e., the time to get absorbed into another strongly connected component. Thus, the total absorbing time $H$ is the sum of the absorbing times of the strongly connected components on the longest path on the condensation of the graph $\mathcal{G}$ from an open strongly connected component to a closed strongly connected component. The condensation of the graph $\mathcal{G}$ is a directed acyclic graph and such path always exist.}

	By the Markov inequality, regardless of where the random walk starts, the probability that it takes more than $4 H$ iterations to get absorbed is at most $1/4$. Thus, for all $k \geq 4H \log (1/ \epsilon)$ steps we have that $\|q_k\|_1 < \epsilon$. 
	
	Now, let $z_{\infty}$ be the vector that satisfies
	\begin{align}\label{eq3}
	z_{\infty} = Z z_{\infty} + Ry_{\infty},
	\end{align}
	which we know exists since every eigenvalue of $Z$ must be strictly less than $1$ (since $Z^k \to 0$). If we define
	\begin{align*}
	\Delta_k = x^M_k - z_{\infty},
	\end{align*}
	then subtracting the updates of $x_M$ and $z_{\infty}$,
	\begin{align}
	\Delta_{k+1}= Z \Delta_k +  R(y_k - y_{\infty}). 
	\end{align}
	
	It follows that $\Delta_k$ goes to zero since we have assumed that $y_k \to y_{\infty}$, and $Z^k \to 0$.

	In conclusion, this argument shows that for all  $k \geq {4(L+H) \log (1/ \epsilon)}$ steps every node is within $\epsilon$ of its limiting value.
\end{proof}

\cu{Lemma~\ref{theorem:time} states that the convergence time of~\eqref{consensus} can be bounded by the absorbing time of a random walk on the graph $\mathcal{P}$ into a closed strongly connected component, in addition to the mixing time of that particular component. Moreover, the mixing time of a closed strongly connected component can be bounded by its coupling time, i.e., the time needed for two independent random walks, with arbitrary initial points, to intersect~\cite{lev09}.}
	
	Figure~\ref{fig:random_walk} illustrates the result in 	Lemma~\ref{theorem:time} by considering two random walks $X$ and $Y$ with the same transition matrix. Assuming the graph $\mathcal{P}$ is aperiodic, we denote by $L$ the maximum expected mixing time among all closed strongly connected components, and by $H$ the maximum expected time to get absorbed into a closed component. Then the belief system will be $\epsilon$ close to its limiting distribution after $O((L+H)  \log (1/\epsilon))$ steps. Therefore, not only do we have an estimate of the convergence time of the belief system in terms of the topology of the graph $\mathcal{P}$, but we also know this convergence happens exponentially fast. For example, in Fig.~\ref{fig:ex1}, the expected absorbing time is of the order of the number of nodes in the path, that is $m$, while the expected mixing time of a cycle graph is of the order of 
	the number of the nodes squared\cite{ger11,tah07,lev09}, which is $n^2$ in this example. 
	Thus, the convergence time for the belief system is $O(\max(n^2,m) \log (1/\epsilon))$. 
	Figure~\ref{fig:example1} depicts simulation results for this bound that demonstrate its validity. 
	In particular, Fig.~\ref{fig:example1}(a) shows how the convergence time changes when the number of nodes in the cycle graph increases, while Fig.~\ref{fig:example1}(b) shows how the convergence time changes when the number of truth statements in the directed path graph increases. Moreover, Fig.~\ref{fig:example1}(c) shows that the convergence to the final beliefs is exponentially fast.
	
	Lemma~\ref{lemma:kron} shows that each strongly connected component of the graph $\mathcal{P}$ is the product of two such components, one from the graph $\mathcal{G}$ and the other from the graph $\mathcal{T}$. Consequently, the next lemma shows that the expected mixing (or absorbing) time for a random walk on a product graph is the maximum of the expected 
	mixing (or absorbing) time of the individual factor graphs. 
	
	\begin{lemma}\label{lemma:mixing}
		{Consider two aperiodic strongly connected directed graphs $\mathcal{G}_1$ and $\mathcal{G}_2$. 
			The expected coupling time of two random walks on the graph $\mathcal{G}_1 \otimes \mathcal{G}_2$ is $L = 8\max\{L_1,L_2\}$, where $L_1$ and $L_2$ are the expected coupling times for random walks on the graphs $\mathcal{G}_1$ and $\mathcal{G}_2$ respectively. Similarly, a random walk on an open strongly connected component of a graph  $\mathcal{G}_1 \otimes \mathcal{G}_2$ has an expected absorbing time (into another strongly connected component) of $H = 8\max\{H_1,H_2\}$, where $H_1$ and $H_2$ are the expected absorbing times for random walks on the graphs $\mathcal{G}_1$ and $\mathcal{G}_2$ respectively. }
	\end{lemma}

\begin{proof}
	Say both graphs $\mathcal{G}_1$ and $\mathcal{G}_2$ are aperiodic and strongly connected, their product is also aperiodic and strongly connected and there exists a limiting distribution $\pi$ for a random walk moving on the Kronecker product graph $\mathcal{G}_1 \otimes \mathcal{G}_2$. 
	
	Consider a random walk $X = (X_k)_0^{\infty}$, on the graph $\mathcal{G}_1 \otimes \mathcal{G}_2$, with transition matrix $A_1 \otimes A_2$ starting with some arbitrary distribution $v$, where $A_1$ is the transition probability on a random walk on the graph $\mathcal{G}_1$ and $A_2$ is the transition probability on a random walk on the graph $\mathcal{G}_2$. Moreover, from the definition of the Kronecker product of graphs, we have that the state space of  $\mathcal{G}_1 \otimes \mathcal{G}_2$ is the Cartesian product $V = V_1 \times V_2$, composed by the ordered pairs $(i,j)$ for $i \in V_1$ and $j\in V_2$. Thus, the probability that the random walk $X$ jumps from the node $(i,j)$ to the node $(\bar i ,\bar j)$ is $[A_1]_{i, \bar i} [A_2]_{j, \bar j} $.
	
	Following the coupling method, define another random walk  $Y = (Y_k)_0^{\infty}$ with the same transition matrix $A_1 \otimes A_2$ but starting at the stationary distribution $\pi$. Now, construct an new random walk as follows:
	\begin{align*}
	W_k = \begin{cases}
	Y_k, & \text{if}\ k <K, \\
	X_k, & \text{if}\ k \geq K ,
	\end{cases}
	\end{align*}\
	where  $K = \min\left\lbrace k \geq 0 : Y_k = X_k\right\rbrace $. Clearly, if the state of the random walk $X$ at time $k$ is $X_k = (i_k, j_k)$ and the state of the random walk $Y$ at time $k$ is $Y_k = (\bar i_k, \bar j_k)$, then the condition $Y_k = X_k$ implies that  $i_k = \bar i_k $ and $j_k = \bar j_k$. Thus, the coupling time $K$ can alternatively be expressed in terms of the two separate conditions  $i_k = \bar i_k $ and $j_k = \bar j_k$, which in turn represents the coupling conditions for two separate random walks on each individual coordinate where each coordinate represents one of the factor graphs. Therefore, we write the coupling time between the random walks $X$ and $Y$ as $K = \min\left\lbrace k \geq 0 : Y_k = X_k\right\rbrace  = \min\left\lbrace k \geq 0 : i_k = \bar i_k ,j_k = \bar j_k \right\rbrace $ which is equivalent to
	\begin{align*}
	K & = \min\left\lbrace k \geq 0 : Y_k = X_k\right\rbrace  \\
	&= \min\left\lbrace k \geq 0 : i_k = \bar i_k ,j_k = \bar j_k \right\rbrace \\
	& = \max \left\lbrace  \min\left\lbrace k \geq 0 : i_k = \bar i_k \right\rbrace, \min\left\lbrace k \geq 0 : j_k = \bar j_k \right\rbrace \right\rbrace  \\
	& = \max \{K_1,K_2\},
	\end{align*}
	where $K_1$ and $K_2$ are the coupling times for the graphs $\mathcal{G}_1$ and $\mathcal{G}_2$ respectively. Thus,
	\begin{align*}
	\mathbb{P}\left\lbrace K > k\right\rbrace & = \mathbb{P}\left\lbrace \max \{K_1,K_2\} > k\right\rbrace \\
	& \leq \mathbb{P}\left\lbrace K_1 \geq k\right\rbrace +  \mathbb{P}\left\lbrace K_2 \geq k\right\rbrace,
	\end{align*}
	where the last inequality follows from the union bound.
	
	Note that given that the initial state of the random walk $X$ is $v$, the random walks on each of its coordinates have some well defined initial state, $v_1(i) = \sum\limits_{j\in V_2} v((i,j))$ and   $v_2(j) = \sum\limits_{i\in V_2} v((i,j))$, where $v_1(i)$ is the probability of staring in node $i \in V_1$, $v_2(j)$ is the probability of starting in node $j \in V_2$, and $v((i,j))$ is the probability of the random walk $X$ to start in the node $(i,j)$.
	
	It follows from Theorem 5.2 in Levin et. al.\cite{lev09} that
	\begin{align*}
	\|v'(A_1 \otimes A_2)_S - \pi\|_{TV} & \leq  \max_{v} \mathbb{P}\left\lbrace K > k\right\rbrace \\
	&\leq  \max_{v_1 } \mathbb{P}\left\lbrace K_1 > k\right\rbrace +  \max_{v_2 } \mathbb{P}\left\lbrace K_2 > k\right\rbrace \\ 
	&\leq  \max_{v_1 } \frac{\mathbb{E}[K_1]}{k} +  \max_{v_2 }\frac{\mathbb{E}[K_2]}{k}\\ 
	&=  \max_{v_1 } \frac{L_1}{k} +  \max_{v_2 }\frac{L_2}{k}. 
	\end{align*}
	
	Thus,  in order to be at a distance at most $1/4$ from the stationary distribution we require $k \geq 8\max \{L_1,L_2\}$. Moreover, in order to be $\epsilon$ close to the stationary distribution we require at least $k \geq 8\max \{L_1,L_2\} \log (1/\epsilon)$ steps in the random walk for any initial state $v$. Finally, the coupling time of $X$ is $L = O(\max\{L_1,L_2\})$.
	
	A similar argument follows for the absorbing time of a random walk on a transient component defined by a product graph requires both coordinates be absorbed individually, thus $H = O(\max\{H_1,H_2\})$.	
\end{proof}

	Lemmas~\ref{theorem:time}~and~\ref{lemma:mixing} provide an explicit characterization of the convergence time in terms of the components of the network of agents and the network of logic constraints. Thus allows us to state our graph-theoretic result on the convergence rate of a belief system with logic constraints.

	\begin{theorem}\label{thm:how}
		Assume the process~\eqref{consensus} converges to equilibrium. Moreover, let $L_{\mathcal{T}}$ and $H_{\mathcal{T}}$ be the maximum expected coupling time and maximum absorbing time of the closed aperiodic and strongly connected components of the graph $\mathcal{T}$, and let $L_{\mathcal{G}}$ and $H_{\mathcal{G}}$ be the maximum expected coupling time and absorbing time of a closed aperiodic and strongly connected components of the graph $\mathcal{G}$ composed by oblivious agents only. Then, for $k \geq 32(\max\{L_{\mathcal{T}},L_{\mathcal{G}}\}+\max\{H_{\mathcal{T}},H_{\mathcal{G}}\}) \log( 1 /\epsilon)$, it holds for the belief system in~\eqref{consensus} that  $\|x_k - x_\infty\|_{TV} \leq \epsilon$.	
	\end{theorem}

\begin{proof}
	The proof follows from Lemmas~\ref{lemma:kron},~\ref{theorem:time},~\ref{lemma:mixing}
\end{proof}

	Table~\ref{tab:summary_results} presents the estimates for the convergence time for belief systems composed of well-known classic graphs, see Supplementary Fig.\ref{fig:ex_torus3_star2} for plots of some of these common graphs. We use the existing results about the mixing time for these graphs (see Supplementary 
	Table~\ref{tab:summary_results2} for a detailed list of references on each of the studied graphs) to provide an estimate of the convergence time of the resulting belief system when all agents are oblivious. \cu{When available, we present tighter upper bounds for the mixing times on strongly connected components derived with methods other than coupling.} Particularly, our method allows the direct characterization of the dynamics of a belief system when large-scale complex networks are involved. For example, we provide convergence time bounds for the case where networks follow random graph models, namely: the geometric random graphs, the Erd\H{o}s-R\'enyi random graphs, and the Newman-Watts small-world networks. These graphs are usually considered for their ability to represent the behavior of complex networks encountered in a variety of fields\cite{wat99,bar03,gan09,bor06} (see Supplementary Fig.~\ref{fig:random_graphs}).

	Figure~\ref{fig:example3} shows experimental results for the convergence time of a belief system for a subset of the graphs given in Table~\ref{tab:summary_results}. For every pair of graphs, we show how the convergence time increases as the number of agents or the number of truth statements change. One can particularly observe the maximum-like behavior on the convergence time as predicted by the theoretical bounds in Theorem \ref{thm:how}. See Supplementary Fig.~\ref{fig:complexities} and Supplementary Fig.~\ref{fig:rand_complexities} for additional numerical results on other combinations of graphs from Table~\ref{tab:summary_results}, and Supplementary Fig.~\ref{fig:convergence} and Supplementary Fig.~\ref{fig:rand_conv} for their linear convergence rates, respectively.
	
	Finally, the next theorem describes how the existence of a clique of a well-connected subset of nodes can guarantee fast mixing of a random walk on a graph.
	
	\begin{theorem}\label{theo:influential}
		Consider a random walk on a connected undirected and static graph $\mathcal{G} = (V,E)$ with $|V| = n$ nodes, and assume there is a subset $\bar V \subset V$ with $M$ nodes such that after $K$ steps, the probability of being in any node in $\bar V$ is at least $\frac{1}{5M}$. Then the mixing time of the corresponding Markov chain is of the order $O(MK \log (1/\epsilon))$.
	\end{theorem}
	
	\begin{proof}
		The proof follows immediately since any two random walks will intersect with probability $\frac{1}{M}$ every $K$ steps.
	\end{proof}

\subsection*{Where Does a Belief System Converge?}

	So far we have discussed the conditions for convergence of a belief system and the corresponding convergence time. Convergence implies the existence of a vector $x_\infty$ where the set of beliefs settles as the number of interactions increases. Particularly, Proskurnikov and Tempo~\cite{pro17} characterize the limiting distribution as a solution of 
	\begin{align*}
	X_{\infty} & = \Lambda A X_{\infty} C' + (I - \Lambda)X_0,
	\end{align*}
	which can be intractable to compute when the matrices $A$ and $C$ are large. We are interested in a characterization of this limit vector that admits a rapid computation of its value.
	
	Lemma~\ref{lemma:kron} shows that one can always group the nodes in the graph $\mathcal{P}$ into open and closed strongly connected components. Moreover, their convergence value depends on whether a node is part of a closed or open strongly connected component of the graph. Thus, our result on the convergence value of a belief system will be stated for nodes in closed or open strongly connected components. However, we start by introducing some notation. Define $x_k^S$ as the vector obtained from $x_k$ by taking only the components of $x_k$ corresponding to the nodes in the set $S$. Moreover, let $P_S$ be the minor of the matrix $P$ obtained by taking into account only the nodes in the set $S$. Then, $P_S$ corresponds to the transition matrix of an irreducible and aperiodic Markov chain with a stationary distribution $\pi_S$, where $\pi_S' P_S = \pi_S'$. The vector $\pi_S$ is effectively the left-eigenvector of the matrix $P_S$ corresponding to the eigenvalue $1$~\cite{lev09}.
	
	Particularly, for nodes in a closed strongly connected component $S$, it follows that $x_{k+1}^S = P_S x^S_{k}$. Whereas for open strongly connected component $S$, we have that $x^S_{k+1} = Z x^S_k + Rx^{S_M}_k$, where $Z$ is strongly connected and substochastic, meaning some rows add up to less than $1$, and $R$ represents how the incoming nodes to the set $S$, denoted by $S^M$, influence the nodes in $S$.
	
	\begin{theorem}\label{thm:where}
		Assume the process~\eqref{consensus} converges to equilibrium. Let $S$ be a strongly connected component of the system graph $\mathcal{P}$, with factors $A_S$ and $C_S$, i.e.,  $P_S = A_S\otimes C_S$. If $S$ is closed, then,
			\begin{align*}
			\lim_{k \to \infty} x_k^i  & = (\pi_{A_S} \otimes \pi_{C_S})' (x_0^{A_S} \otimes x_0^{C_S}) \qquad \forall i\in S.
			\end{align*}
		Moreover, if $S$ is open with edges coming from the set of nodes $S^M$, then
		\begin{align*}
		\lim_{k \to \infty} x_k^i = \sum_{j \in S^M} p_{ij} x^{j}_\infty \qquad \forall i \in S, 
		\end{align*}
		where $p_{ij}$ is the probability of absorption of a random walk starting at node $i$ into a node $j \in S^M$ with limiting value $x^j_\infty$. 
	\end{theorem}

\begin{proof}
	It follows form Lemma~\ref{lemma:kron} that every strongly connected component of $\mathcal{P}$ is the product of two strongly connected components, one from the network of agents and one from the logic constraint network. Thus, \mbox{$P_S = A_S \otimes C_S$} for some matrices $A_S$ and $C_S$ (sub-matrices of $A$ and $C$ respectively),
	which implies that $\pi_S = \pi_{A_S} \otimes \pi_{C_S}$, i.e., the vectors $\pi_S^A$ and $\pi_S^C$ are the corresponding left eigenvalues of the factor components of $P_S$ associated with the eigenvalue $1$. 
	
On the other hand, without loss of generality, assume that $\lim_{k \to \infty}x_{k}^{S_M} = x_{\infty}^{S_m}$. Therefore, we can analyze the dynamics in the open strongly connected component $S$ as follows: Initially define the following two systems
	\begin{align*}
	\bar x^S_{k+1} & = Z \bar x^S_k + Rx_{\infty}^{S_M}, \qquad \text{and} \qquad x^S_{k+1} = Z x^S_k + Rx_{k}^{S_M}.
	\end{align*}
	
	It follows that 
	\begin{align*}
	\lim_{k \to \infty} (\bar x^S_{k+1} - x^S_{k+1} ) & = Z  \lim_{k \to \infty} (\bar x^S_{k} - x^S_{k} ) + R \lim_{k \to \infty} (x_{\infty}^{S_M},  -x_{k}^{S_M} ) =
 Z  \lim_{k \to \infty} (\bar x^S_{k} - x^S_{k} ) .
	\end{align*}
	Moreover, given that $Z$ is substochastic, the magnitude of its eigenvalues are strictly less than $1$ and $1-Z$ is invertible. Thus, we can conclude that $\lim_{k \to \infty} \bar x^S_{k} = \lim_{k \to \infty}x^S_{k} $.
	
	Stacking the vector $\bar x^S_k$ and $x_{\infty}^{S_M}$ into a single vector we obtain the following recursion:
	\begin{align*}
	\left[ \begin{array}{c c}
	\bar x^S_{k+1}  \\
	x_{\infty}^{S_M}
	\end{array}\right]  & = Q_S \left[ \begin{array}{c c}
	\bar x^S_{k}  \\
	x_{\infty}^{S_M}
	\end{array}\right], \qquad \text{where} \qquad
	Q_S  = 
	\left[ \begin{array}{c c}
	Z & R \\
	0 & I
	\end{array}\right].
	\end{align*}
	
	Thus, in order to find the limit value of the set of beliefs in $S$ we can focus on the analysis of the powers of the matrix $Q^S$. 
	
	We have that 
	\begin{align*}
	\lim_{k \to \infty} Q_S^{k} & = 
	\left[ \begin{array}{c c}
	0 & N R \\
	0 & I
	\end{array}\right] ,
	\end{align*}
	where $N = I + Z + Z^2 + \cdots = (1 - Z)^{-1}$. The matrix $NR$ is the absorbing probability matrix, see Chapter~$3$ in Kemeny and Snell~\cite{kemeny1983finite}, where $p_{ij} \triangleq [NR]_{ij}$ is the probability of being absorbed by into the node $j \in S^M$ starting from node $i \in S$. Moreover, it follows that for any node $i\in S$ 
	\begin{align*}
	\lim_{k \to \infty} x_k^i = \sum_{j \in S^M} p_{ij} x^{j}_\infty. 
	\end{align*}
\end{proof}
	
  Theorem~\ref{thm:where} shows that the final beliefs of those nodes in closed strongly connected components are a weighted average of the initial beliefs in that component, and the weights (sometimes referred to as the social power) are determined by the product of the left-eigenvectors of the factors $A_S$ and $C_S$. On the other hand, the limiting value of nodes in an open strongly connected components is a convex combination of the limiting values of the nodes from incoming edges. Moreover, their weights are defined by the absorption probabilities.

\subsection*{Numerical Analysis of Social Networks}

	Next, we provide a numerical analysis for the evolution of belief systems with social network structures from \textit{large-scale networks} in the Stanford Network Analysis Project (SNAP)\cite{graphs_data}, see Fig.~\ref{fig:graphs_real}, and logic constraints built from random graph generating models. Random graph generating models, such at the Erd\H{o}s-R\'enyi graphs, the Newman-Watts graph, and the geometric random graphs, have been proposed to model the dynamics and the properties of real large-scale complex networks, for example, relatively fast mixing or linear convergence of the beliefs. We use the \texttt{wiki-Vote}~\cite{les10}, \texttt{ca-GrQc}~\cite{les07}, and \texttt{ego-Facebook}~\cite{les12} graphs as social networks and a binary tree, a Newman-Watts graph, and an Erd\H{o}s-R\'enyi graphs as logic constraints. 
	
	The \texttt{wiki-Vote} network represents the aggregation of $2794$ elections where $7115$ Wikipedia contributors assign votes to each other to select administrators. This generates a directed social network where the edges are the votes given by the users. The \texttt{ca-GrQc} network represents the general relativity and quantum cosmology collaboration network for e-prints from arXiv. The nodes are composed of $5242$ authors, and edges represent co-authorship of a manuscript between two authors. Finally, the \texttt{ego-Facebook} network represents an anonymized set of Facebook users as nodes and edges indicate friendships among them in the Facebook platform. Table~\ref{tab:graphs} shows the description of the networks used. In the three cases, we select the largest strongly connected component of the graph and use it as a representative of the network structure and the mixing properties of the graph. Furthermore, we assume that the agents use equal weights for all their (in)neighbors.

	Figure \ref{fig:mix_real} shows the convergence time of a belief system when the network of agents is each the three large-scale complex networks described in Table \ref{tab:graphs}. Figure \ref{fig:mix_real} considers a simplified scenario where a single closed strongly connected component composes the social network of agents and the network of logic constraints. Therefore, absorbing time is effectively zero and the mixing time of the belief system is the maximum between the mixing time of the social network and the mixing time of the network of logic constraints. Convergence is guaranteed since both networks are taken to be aperiodic by introducing positive self-weights to every agent. Results show that the predicted maximum type behavior holds; that is, the convergence time of the belief system is upper bounded by the maximum mixing time of a random walk on the graph of agents and the graph of logic constraints. The convergence time remains constant and of the order of the convergence time of the network of agents, until the mixing time of the network formed by the logic constraints is larger. Then, the total convergence time increases based on the specific topology of the graph of logic constraints. Figure~\ref{fig:ex4_distance} shows the exponential convergence rate of the belief system described in Figure~\ref{fig:mix_real}. 

\section*{Discussion}

In a recent paper, Friedkin et al. \cite{fri16} proposed a new model that integrates logic constraints into the evolution opinions of a group of agents in a belief system. Logic constraints among truth statements have a significant impact on the evolution of opinion dynamics. Such restrictions can be modeled as graphs that represent the favorable or unfavorable influence the beliefs on specific topics have on others. Starting from this context, we have here approached this model from its extended representation of a belief system, where opinions of all agents on all topics as well as their corresponding initial values are nodes in a larger graph. This larger graph is composed of the Kronecker product of the graphs corresponding to the network of agents and the network of logical constraints respectively.
	
	In this study, we have provided graph-theoretic arguments for the characterization of the convergence properties of such opinion dynamic models based on extensive existing knowledge of convergence and mixing time of random walks on graphs using the theory of Markov chains. We have shown that convergence occurs if every strongly connected component of the network of logic constraints is aperiodic and every strongly connected component of oblivious agents is aperiodic as well. Moreover, to be arbitrarily close to their limiting value we require $O((L+H)\log (1 / \epsilon))$ time steps. The parameter $L$ is the maximum coupling time for a random walk among the closed strongly connected components of the product graph, and $H$ is the maximum time required for a random walk, that starts in an open component, to get absorbed by a closed component. Our analysis applies to broad classes of networks of agents and logic constraints for which we have provided bounds regarding the number of nodes in the graphs. Finally, we show that the limiting opinion value is a convex combination of the nodes in the closed strongly connected components and this convergence happens exponentially fast. 
	
	Our framework offers analytical tools that deepen our abilities for modeling, control and synthesis of complex network systems, mainly human-made, and can inspire further research in domains where opinion formation and networks interact naturally, such as neuroscience and social sciences. Finally, extending this analysis to other opinion formation models that use different aggregating strategies may require further study of Markov processes and random walks.

\section*{Methods}

\subsection*{Directed Graphs~\cite{McAndrew1963}}
	
	 We define a directed graph $\mathcal{G}$ as a set of nodes $V$ and a set of edges $E$ where the elements of $E$ are ordered pairs $(j,i)$ with $i,j \in V$. A path $\mathbf{P}$ of $\mathcal{G}$ is a finite sequence $\{p_i\}_{i=0}^{l}$ such that $(p_i,p_{i+1}) \in E$ for $0 \leq i \leq l-1$. Moreover, define $n(\mathbf{P})$ as the number of edges in the path $\mathbf{P}$. A pair of nodes $(i,j)$ are \textit{strongly connected} if there is a path from $i$ to $j$ and from $j$ to $i$. We say a directed graph $\mathcal{G}$ is \textit{strongly connected} if each pair of nodes of $\mathcal{G}$ are strongly connected. A cycle $\mathbf{C}$ of a graph $\mathcal{G}$ is a path $\mathbf{P}$ such that $p_0 = p_l$, i.e., the start and end nodes of the path are the same.  We denote the \textit{period} of a directed graph as $d(\mathcal{G})$, and define it as the greatest common divisor of the length of all cycles in the graph $\mathcal{G}$.
	
	\subsection*{Random Walks, Mixing and Markov chains}\label{method:mix}
	
	Consider a finite directed graph $\mathcal{G}=(V,E)$ composed by a set $V$ of nodes with a set of edges $E$ and a compliant associated row-stochastic matrix $P$, called the transition matrix. A random walk on the graph $\mathcal{G}$ is the event of a token moving from one node to an out-neighbor according to some probability distribution determined by the transition matrix. The dynamics of the random walk are modeled a Markov chain $X = (X_k)_0^{\infty}$ such that $\mathbb{P}\{X_{k+1}=j|X_k = i\}= P(i,j)$ with $i,j \in V$. This Markov chain is called \textit{ergodic} if it is irreducible and aperiodic. For an ergodic Markov chain, there exists a unique stationary distribution $\pi$, which describes the probability that a random walk visits a particular node in the graph as the time goes to infinity, that is $\mathbb{P}\{X_k = j\} \to \pi_j$ as $k \to \infty$. The stationary distribution is invariant for the transition matrix, that is $\pi' P = \pi'$. It follows that the convergence to the stationary distribution of a random walk reduces to analyzing powers of $P$ (Theorem $4.9$ in Levin et al.\cite{lev09}).
	
	The distance to stationarity at a time $k$, i.e., after $k$ transitions of the Markov Chain, or $k$ steps in the random walk, is defined as
	\begin{align*}
	d(k) & = \max_{x \in \Omega} \|P^k(x,\cdot) - \pi\|_{TV},
	\end{align*}
	where $\|\mu - \nu\|_{TV}$ is the \textit{total variation distance} between two probability distributions $\mu$ and $\nu$, defined as
	\begin{align*}
	\|\mu - \nu\|_{TV} & = \sup_{{\rm events}~ A \in \mathcal{F}} | \mu(A) - \nu(A) |.
	\end{align*}
	
	Moreover, the mixing time of the Markov chain is
	\begin{align*}
	t_{\text{mix}}(\epsilon) & = \min_{k\ge 0} \{k : d(k) \leq \epsilon \},
	\end{align*}
	and we say the Markov chain has (relatively) rapid mixing if $t_{\text{mix}}(\epsilon)  = \text{poly}(\log n, \log \frac{1}{\epsilon})$, i.e., polynomial relations in the terms $\log n$ and $\log(1/\epsilon)$. Finally, the mixing time can be bounded in terms of the left eigenvalues of the matrix $P$ as
	\begin{align}\label{eq:upper}
	\frac{\lambda_2}{2(1-\lambda_2)} \log\left(\frac{1}{2\epsilon} \right) & \leq t_{\text{mix}}(\epsilon) \leq \frac{\log n + \log (1 / \epsilon)}{1 -\lambda_2},
	\end{align} 
	where $\lambda_2$ is the left-eigenvalue of the transition matrix $P$ with the largest abstolute value\cite{dia91}.
	
	\subsection*{The Coupling Method}\label{methods:coupling}
	
	The technical advances in this paper are mostly made by using the coupling method, which is a way to bound the mixing time of Markov chains. Consider two independent Markov chains $X = (X_k)_0^{\infty}$  and  $Y= (Y_k)_0^{\infty}$, with the same transition matrix $P$. Then, define the \textit{coupling time} $K$ as the smallest $k$ such that $X_k = Y_k$, that is, $K = \min_{k\geq0} \{X_k = Y_k\}$. Note that $K$ is a random variable and it depends on $P$ as well as the initial distributions of the processes $X_k$ and $Y_k$. Finally, define the quantity $L$ as the maximum expected coupling time of a Markov chain with transition matrix $P$ over all possible initial distributions of the processes $X_k$ and $Y_k$, i.e.,
	\begin{align*}
	L & = \max_{u,v} \mathbb{E}[K] \qquad \text{where} \qquad X_0 = u \  \text{and}  \ Y_0 = v.
	\end{align*}
	In words, this $L$ is the maximum expected time it takes for two random walks, with the same transition matrix and arbitrary initial states, to intersect. If we assume $X$ starts from a distribution $\pi$, and $Y$ from some other arbitrary stochastic vector $v$ and we \textit{couple} the processes $Y$ and $X$ by defining a new process $W$ such that
	\begin{align*}
	W_k = \begin{cases}
	Y_k, & \text{if}\ k <K \\
	X_k, & \text{if}\ k \geq K 
	\end{cases}
	\end{align*}
	The key insight of the coupling method is that $W_k$ is identically distributed to $X_k$; this follows by conditioning on the events $K \leq k$ and $K > k$. Therefore, questions about the distribution of $X_k$ can be solved by considering $W_k$ instead. 
	
	By starting the chain $X_k$ in the stationary distribution, these considerations imply that 
	\begin{align*}
	\|v'P^k - \pi\|_{TV} & \leq   \mathbb{P}\left\lbrace K > k\right\rbrace, 
	\end{align*} because if $K\leq k$ then $W_k=Y_k$; for more details, see Lindvall~\cite{lin02}. Thus, it follows by the Markov inequality that
	\begin{align*}
	\|v'P^k - \pi\|_{TV} & \leq  \frac{\mathbb{E}[K]}{k}.
	\end{align*} 
	
	Setting $k=2 \mathbb{E}[K]$ implies that \mbox{$\|v'P^k - \pi\|_{TV} \leq 1/2$}. Thus, it follows that after  $T = O(L \log (1/\epsilon))$ steps, it holds that ${{\|v^T P^T - \pi \|_1 \leq \epsilon}}$, for any $v$, and $\pi$ being the stationary distribution of the Markov chain. Since $||p-q||_{TV} = (1/2) ||p-q||_1$~\cite{lev09}, the same applies to the quantity $\|v'P^k - \pi\|_{1}$.

	The coupling method is the primary technical tool we use in this work. In Supplementary Note 3, we use the coupling method to bound the convergence time of equation~\eqref{individual_model} in terms of the coupling times on the underlying social network and on the logic constraint graph.  Because coupling time over the Kronecker product is, up to a multiplicative constant, the maximum of the coupling times, this allows us to analyze the effect of the social network and logic constraint graph on convergence time separately.

	\newpage
	
	\begin{table}[ht]
		\centering
		%    \resizebox{\columnwidth}{!}{
		\begin{tabular}{|ccc|}  
			\hline
			\multirow{ 2}{*}{\textbf{Network of Agents}} & \multirow{ 2}{*}{\textbf{Logic Constraints}}& \multirow{ 2}{*}{\textbf{Convergence Time} $\approx$} \\
			&  & \\ 
			\hline
%			Complete              & Directed Path                & $m $     \\ 
			Cycle     & Directed Path                & $ \max(n^2,m) $        \\ 
			Cycle     & Path                & $\max(n^2,m^2) $        \\ 
			Dumbbell Graph      & Complete Binary Tree      & $\max(n^2,m)$        \\ 
			%$k$-d Cube with Loops    & Complete Binary Tree      & $\max((1-1/k),m$        \\ 
			$k$-d Hypercube $\{0,1\}^k$   & Complete Binary Tree      & $\max(k \log k,m)$        \\ 
			%Lovasz Graph $\mathcal{C}_n^k$     & Dumbbell       & $ \max(1-1/(kn^2),m^2)$        \\ 
			$2$-d Grid            & Star                & $n \log n $     \\ 
			$3$-d Grid            & Two Joined Star    & $n^{2/3} \log n$       \\ 
			$k$-d Grid            & Star              & $ k^2n^{2 / k} \log n $  \\ 
			$2$-d Torus            &     $2$-d Grid                 & $ \max(n^2,m \log m) $  \\ 
			$3$-d Torus            & Star                 & $n^2 $  \\ 
			$k$-d Torus            & $k$-d Grid                & $ \max(n^2k \log k,k^2m^{2 / k} \log m) $  \\ 
			Lollipop            & Star               & $ n^2 $  \\ 
			Dumbbell                & Star                 & $ n^2 $  \\ 
			Eulerian: $d$-degree and expansion                & Dumbbell                & $ \max(|E|^2,m^2) $  \\ 
			Eulerian: $d$-degree, max-degree weights              & Dumbbell                & $ \max(n^2d,m^2) $  \\ 
			Lazy Eulerian with degree $d$-degree          & Dumbbell                & $ \max(n |E|,m^2) $  \\ 
			Lamplighter on $k$-Hypercube        & Bolas        & $  \max(k 2^k,m^3)  $       \\ 
			Lamplighter on $(k,n)$-Torus            & Bolas        & $ \max( k n ^ k,m^3)  $       \\ 
			Geometric Random on $\mathbb{R}^d$: $\mathcal{G}^d(n,r)$        & Bolas        & $  \max(r^{-2}\log n,m^3)  $       \\ 
			Geometric Random: $r =\Omega(\text{polylog}(n))$        & Bolas        & $ \max(\text{polylog}(n),m^3)  $       \\ 
			Erd\H{o}s-R\'enyi: $\mathcal{G}(n,c/n)$, $c>1$            &  Dumbbell   & $\max ( \log^2 n, m^2 ) $       \\ 
			Erd\H{o}s-R\'enyi: $\mathcal{G}(n,c/n)$, $c>1$            &  Newman-Watts   & $\max ( \log^2 n, \log^2 m) $       \\ 
			Erd\H{o}s-R\'enyi: $\mathcal{G}(n,(1+\delta)/n)$, $\delta^3n \to \infty$            &  Dumbbell   & $ \max ( (1/ \delta^3)\log^2 (\delta^3n), m^2 ) )$       \\ 
			Erd\H{o}s-R\'enyi: $\mathcal{G}(n,1/n)$    &  Dumbbell   & $\max ( n, m^2 )$       \\ 
			Newman-Watts : $\mathcal{G}(n,k,c/n)$, $c>0$      &  Path   & $ \max ( \log^2 n, m^2 )  $ \\  
			Expander    & Path                & $ m^2 $        \\  
			%Exponential Random: High temperature     & Path                & $\max( n^2 \log n,m^2) $        \\  
			%Exponential Random: Low temperature     & Path                & $\max(\exp(n),m^2) $        \\  
			Any Connected Undirected Graph   & \multirow{ 2}{*}{Expander}                &\multirow{ 2}{*}{$ n^2 $  }         \\ 
			with Metropolis Weights   &                 &        \\ 
			Any Connected Undirected Graph    & Expander                & $|E| \text{diam}(\mathcal{G}) $        \\  
			\hline
		\end{tabular}
		%    }
		\caption{\textbf{Convergence time for the belief system with logic constraints for different networks of Agents with $n$ nodes and networks of truth statements with $m$ nodes.} The approximated maximum expected convergence time identified as $\approx$ should be understood in terms of the order $O(\cdot)$, that is, an estimate up to constant terms. Additionally, all the estimates provided should be multiplied by the accuracy term $\log (1/\epsilon)$.}
		\label{tab:summary_results}
	\end{table}

	\bgroup
	\def\arraystretch{1.1}
	\begin{table}[ht]
		\centering
		\resizebox{\columnwidth}{!}{
			\begin{tabular}{|cccccp{4cm}|}  
				\hline
				\multirow{ 2}{*}{\textbf{Graph}} & \multirow{ 2}{*}{\textbf{Nodes}} & \multirow{ 2}{*}{\textbf{Edges}}  & \multirow{ 2}{*}{\textbf{Type}} &  \textbf{Upper Bound on} & \multirow{ 2}{*}{\textbf{Description}}     \\
				&               &        &       &\textbf{Mixing Time} &    \\
				\hline 
				\texttt{wiki-Vote} \cite{les10}          & $1300$               & $103663 $        & Directed&$145$ &Wikipedia who-votes-on-whom network \\ 
				\texttt{ca-GrQc}  \cite{les07}           & $4158$               & $13428$          & Undirected& $12308$&Collaboration network of arXiv General Relativity\\ 
				\texttt{ego-Facebook}\cite{les12}        & $3927$               & $88234$          & Undirected& $53546$&Social circles from Facebook\\ 
				\hline
			\end{tabular}
		}
		\caption{\textbf{Datasets of large-scale networks.} Description, number of nodes, number of edges, simulated mixing time and an upper bound on the mixing time of the three datasets used in the numerical analysis. The upper bound on the mixing time is computed from the second largest eigenvalue bound in equation~\eqref{eq:upper}.}
		\label{tab:graphs}
	\end{table}
	
\begin{figure}[ht]
		\centering
		\subfigure[]{
			\includegraphics{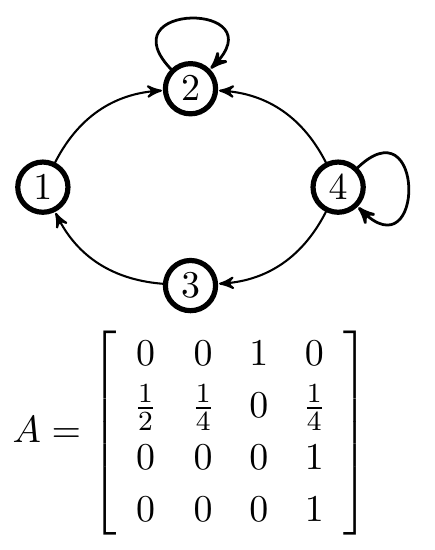}
		}
		\subfigure[]{
			\includegraphics{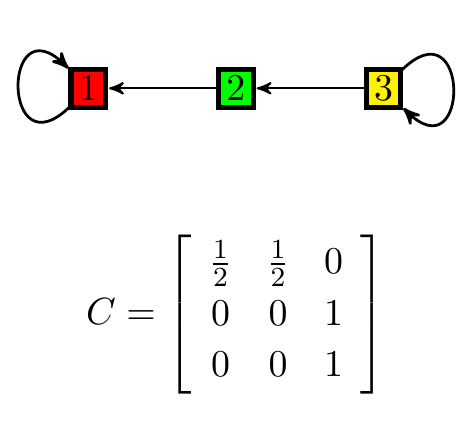}
		}
		\subfigure[]{
			\includegraphics{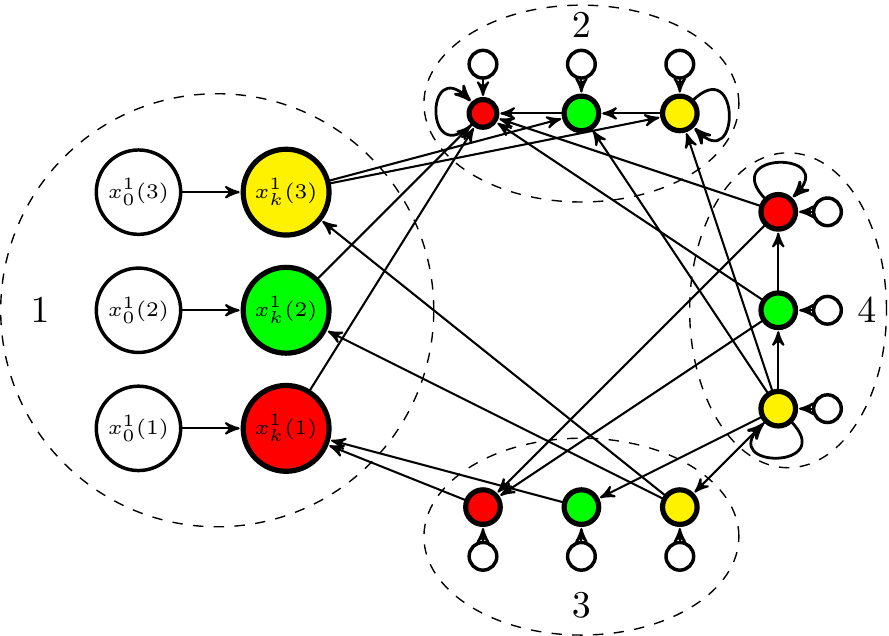}
		}
		\caption{\textbf{A belief system with $4$ agents and $3$ truth statements}. (a) Agents are represented as nodes/circles, numbered from $1$ to $4$, and the network of influences among them is shown as edges between nodes. The truth statements or topics are color-coded, e.g., the truth statement $1$ is represented as a red square. Agent $2$ is influenced by its own opinion and agents $4$ and $1$, agent $1$ follows the opinion of agent $3$ which in turn follows the opinion of agent $4$, agent $4$ follows its own opinion only. A possible matrix $A$ for this social network is shown below the graph. This indicates that agent $2$ assigns a higher weight of $\tfrac{1}{2}$ to the opinion of agent $1$ than the weight it assigns to the opinion of communicated by agent $4$. (b) The truth statement $1$ is influenced by the belief that statement $2$ is true, statement $2$ directly follows the belief in statement $3$. A possible matrix $C$ for this set of logic constraints is shown below the graph. The belief that the truth statement $1$ is true is influenced (with a weight of $\tfrac{1}{2}$) by the opinion that the truth statement $2$ is true. (c) The beliefs system, see equation \ref{consensus}, composed by the agent's interaction graph and the logic constraints.}
		\label{networks1}
	\end{figure}

\begin{figure}[ht]
		\centering
		\subfigure[]{
			\includegraphics{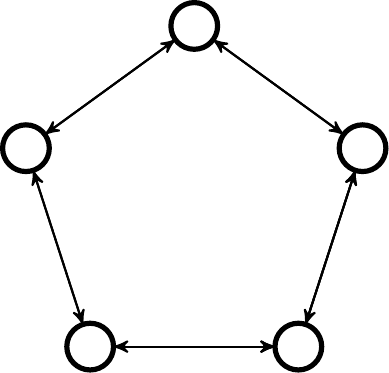}
		}
		\subfigure[]{
			\includegraphics{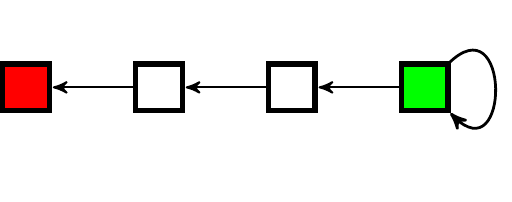}
		}
		\\
		\subfigure[]{
			\includegraphics{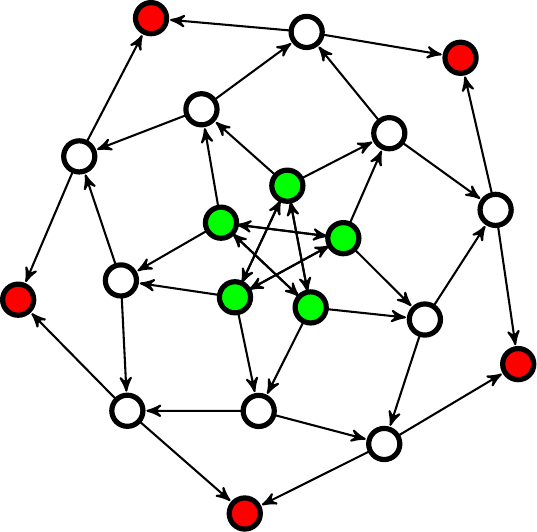}
		}
		\\
		\subfigure[]{
			\includegraphics{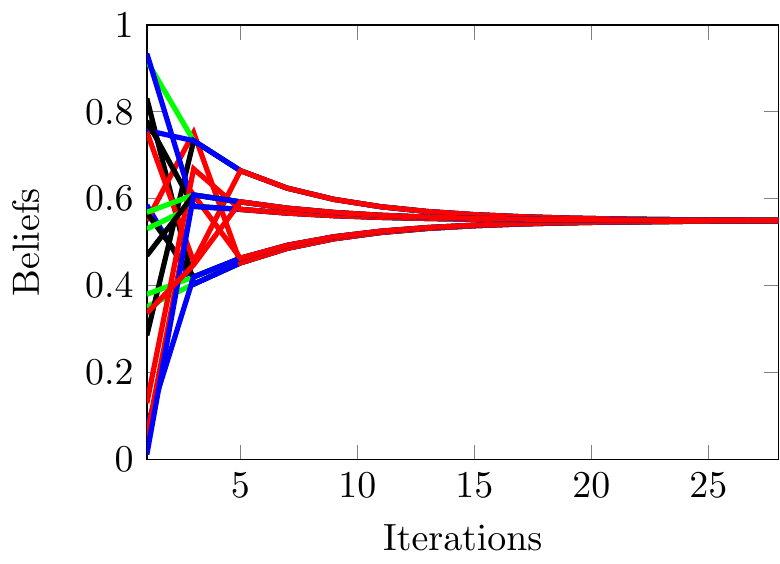}
		}
		\subfigure[]{
			\includegraphics{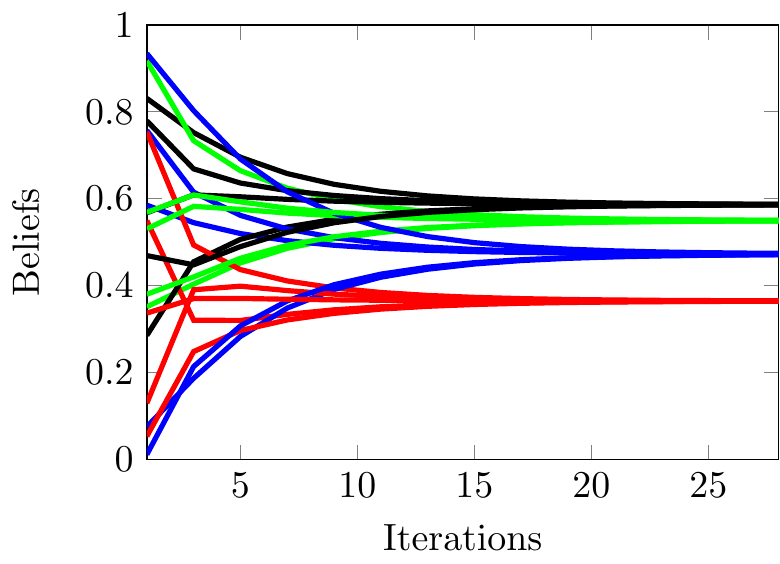}
		}
		\caption{\textbf{A belief system with agents on a cycle graph and logic constraints on a path graph.} (a) A network of $5$ oblivious agents forming a cycle graph. (b) A set of $4$ truth statements with logic constraints forming a path graph. (c) The belief system graph $\mathcal{P}$. (d) The belief dynamics with logic constraints. (e) The belief dynamics with no logic constraints. The beliefs of all agents have been color coded per truth statement. The agents reach an agreement on each of the truth statements.}
		\label{fig:ex1}
	\end{figure}

	\begin{figure}[ht]
		\centering
			\includegraphics{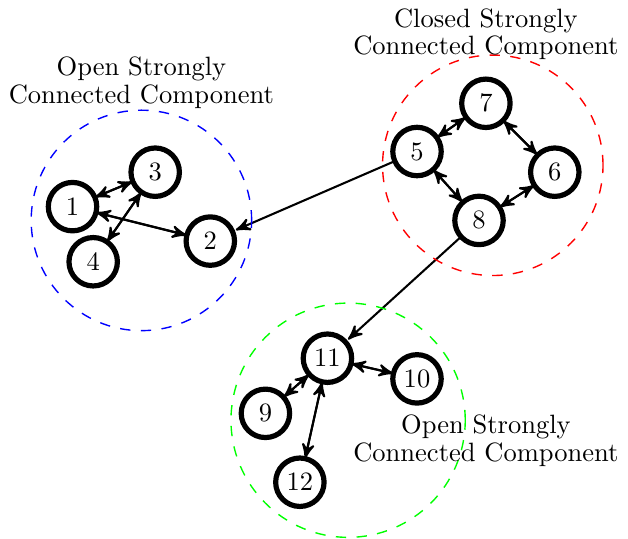}
		\caption{\textbf{Open and closed strongly connected components of a graph}. A graph with $12$ nodes and $3$ strongly connected components. The strongly connected component composed of nodes $5$, $6$, $7$ and $8$ is closed since it has no incoming edges from other components. }
		\label{fig:cond}
	\end{figure}

	\begin{figure}[ht]
		\centering
		\includegraphics{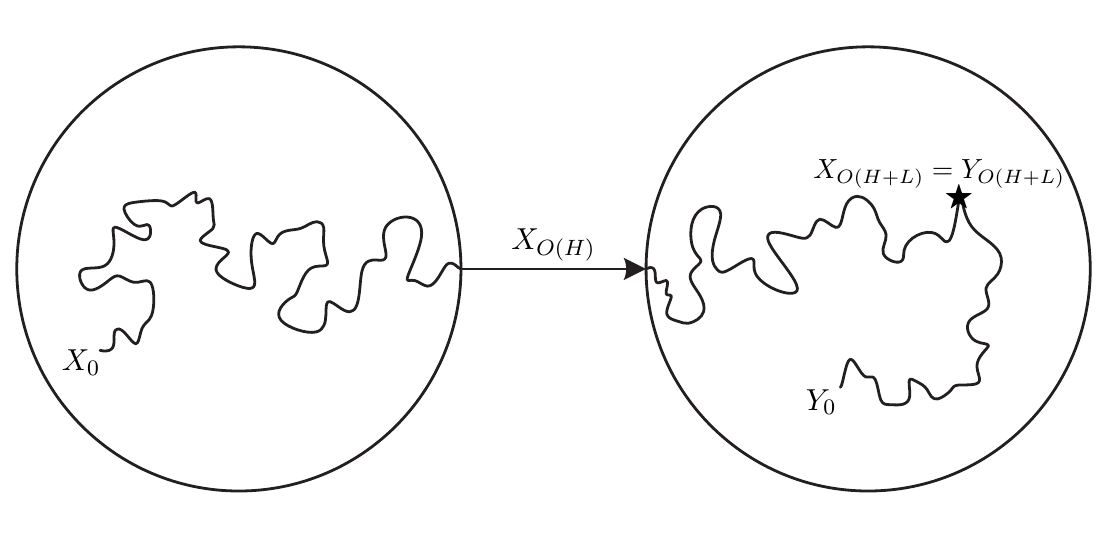}
		\caption{\textbf{Coupling of two random walks.} A random walk starts at $X_0$ in a transient state and evolves according to some transition matrix $P$; after $O(H)$ time steps (the absorbing time), it gets absorbed into a closed connected component. Then, after $O(L)$ time steps (the mixing time) it crosses paths with another random walk $Y_k$ starting at $\pi$ the stationary distribution of $P$. Then after $O((L+H)\log 1 / \epsilon)$ time steps, the random walk $X_0$ is arbitrarily close to its limit value. Note that the random walk moves in the opposite direction to the edges in the graph.}
		\label{fig:random_walk}
	\end{figure}

\begin{figure}[ht]
		\centering
		\subfigure[]{
			\includegraphics{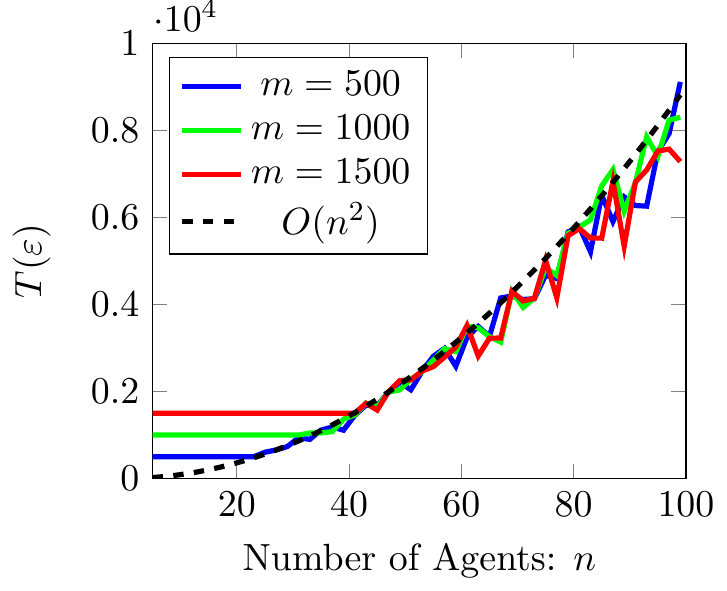}
		}
		\subfigure[]{
			\includegraphics{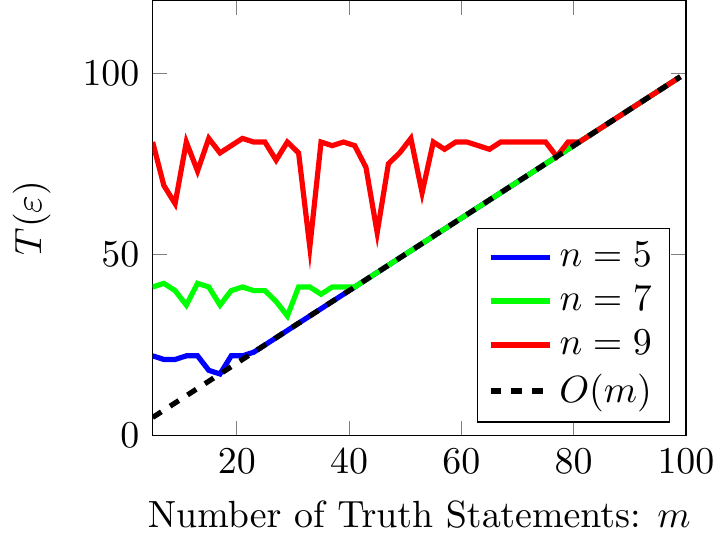}
		}
		\\
		\subfigure[]{
			\includegraphics{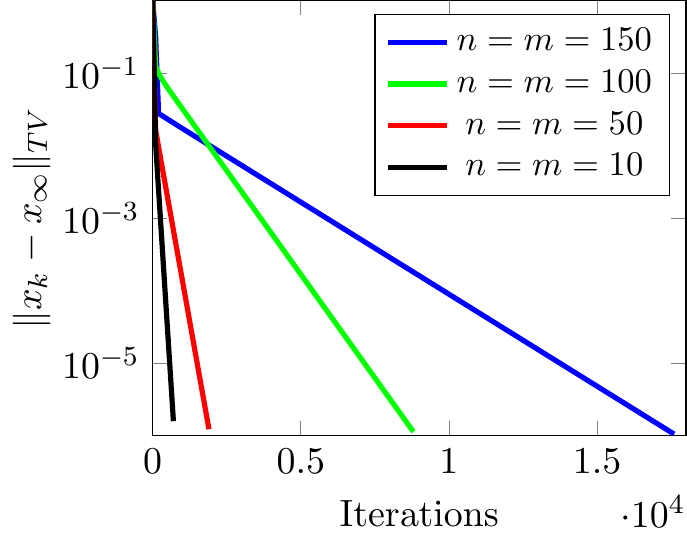}
		}
		\caption{\textbf{Convergence time for a belief system with an undirected cycle as a social network 
				and a directed path as a network for the logic constraints.} (a)~Varying the number of the agents in the social graph. (b)~Varying the number of the truth statements for a directed path. (c)~The exponential convergence rate of the belief system.}
		\label{fig:example1}
	\end{figure}

	\begin{figure}[ht]
		\centering
		\subfigure[]{
			\includegraphics{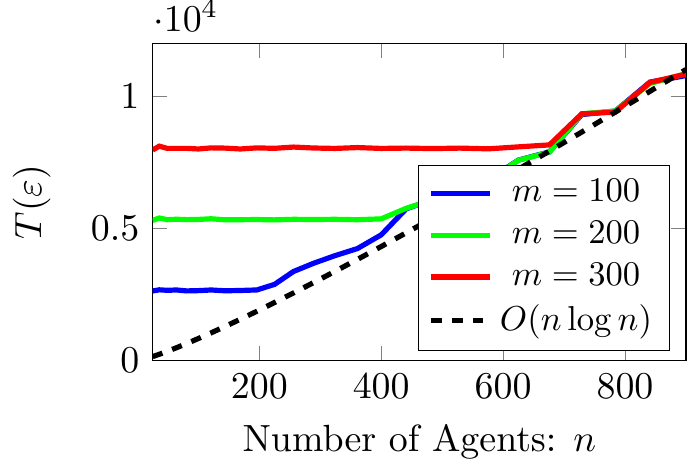}
		}
		\subfigure[]{
			\includegraphics{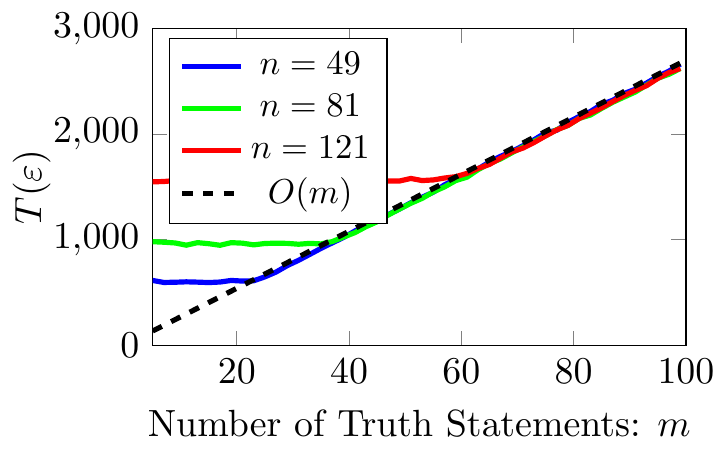}
		}
		\\
		\subfigure[]{
			\includegraphics{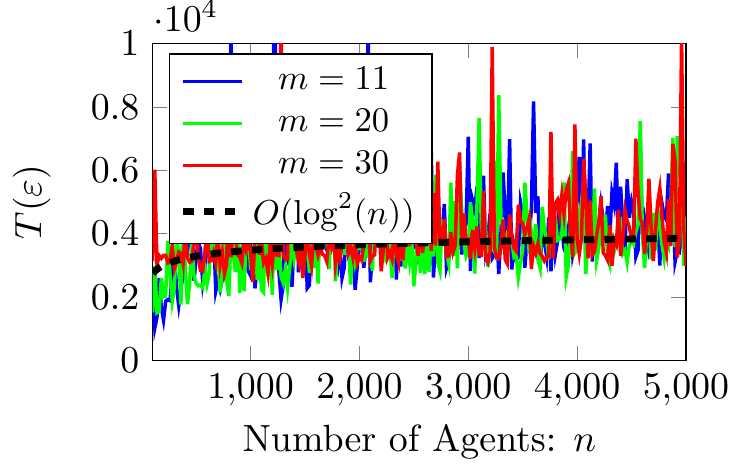}
		}
		\subfigure[]{
			\includegraphics{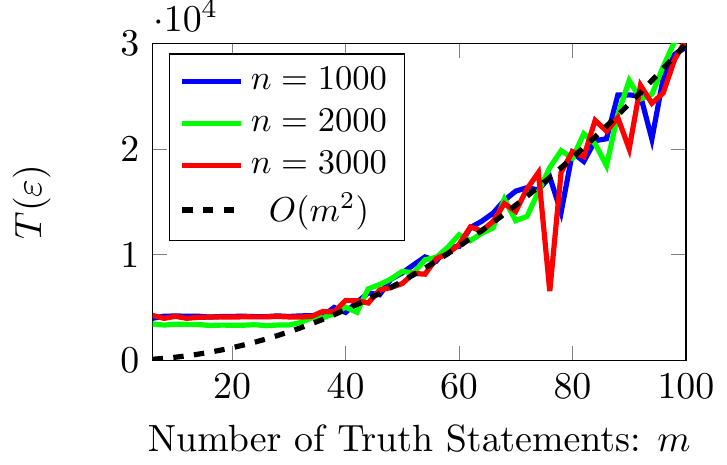}
		}
		\\
		\subfigure[]{
			\includegraphics{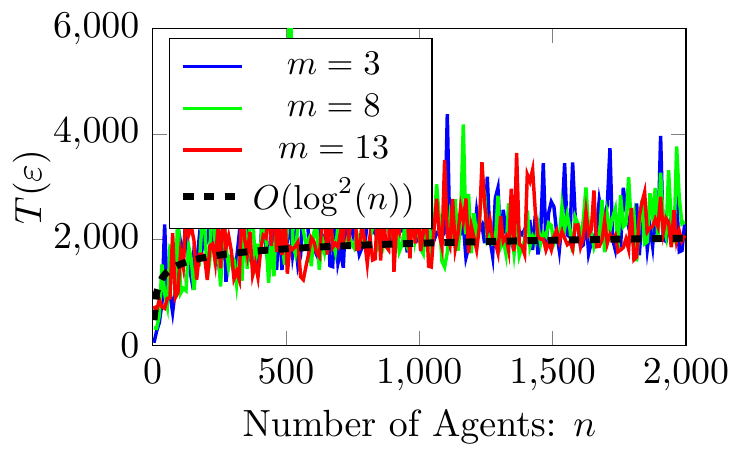}
		}
		\subfigure[]{
			\includegraphics{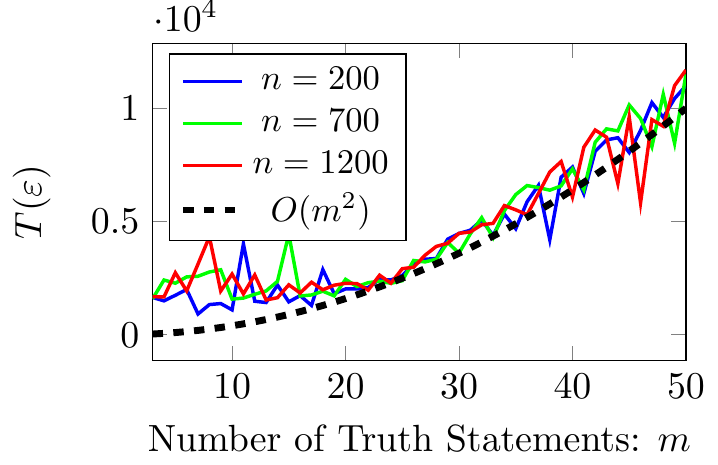}
		}
		\caption{\textbf{Convergence time or various belief systems.} (a) Varying the number of agents on a $2$d-Grid with fixed the number of truth statements on a star graph. (b) Varying the number of truth statements on a star graph with a fixed number of agents on a $2$d-Grid. (c) Varying the number of agents on a \mbox{\protect Erd\H{o}s-R\'enyi} graph with fixed the number of truth statements on a  dumbbell graph. (d) Varying the number of truth statements on a dumbbell graph with a fixed number of agents on a  \mbox{\protect Erd\H{o}s-R\'enyi} graph. (e) Varying the number of agents on a \mbox{\protect Newman-Watts} small-world graph with fixed the number of truth statements on a path graph. (f) Varying the number of truth statements on a path graph with a fixed number of agents on a \mbox{\protect Newman-Watts} small-world graph.}
		\label{fig:example3}
	\end{figure}

	\begin{figure}[ht]
		\centering
		\subfigure[]{
			\includegraphics[trim={3cm 2cm 2.5cm  2cm},clip,width=0.45\textwidth]{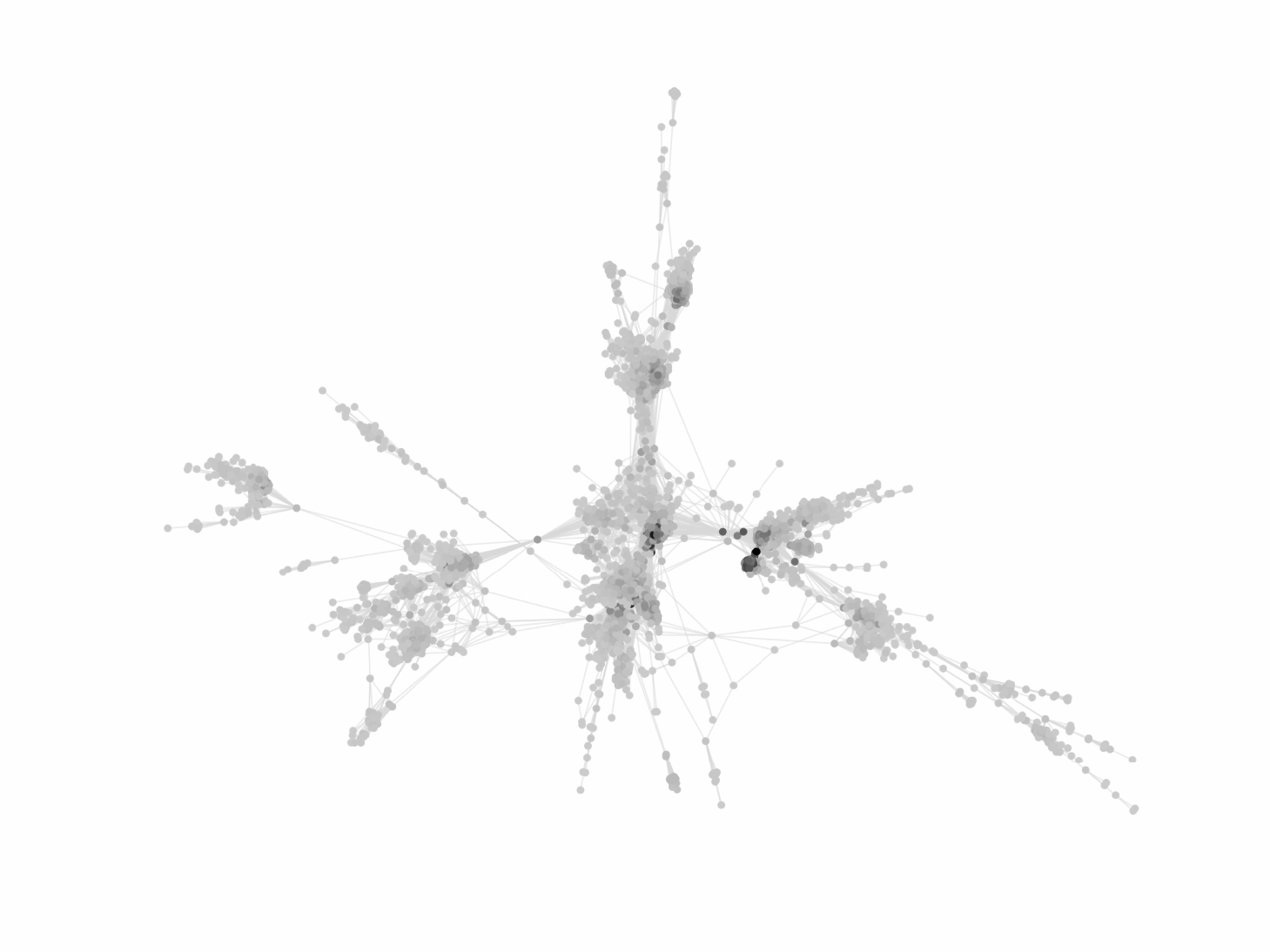}
		}
		\\
		\subfigure[]{
			\includegraphics[trim={3cm 2cm 2.5cm  2cm},clip,width=0.45\textwidth]{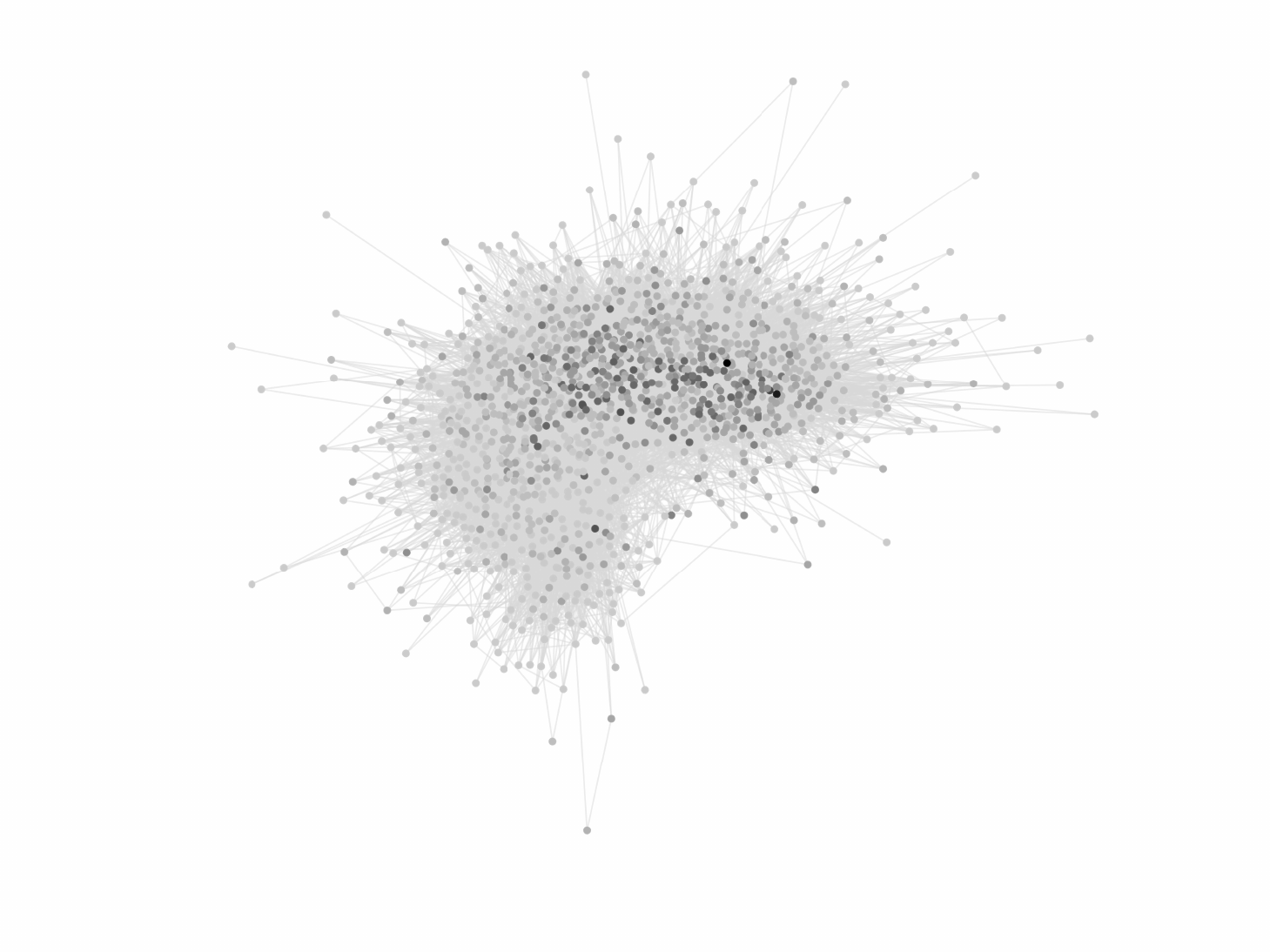}
		}
		\subfigure[]{
			\includegraphics[trim={3cm 2cm 2.5cm  2cm},clip,width=0.45\textwidth]{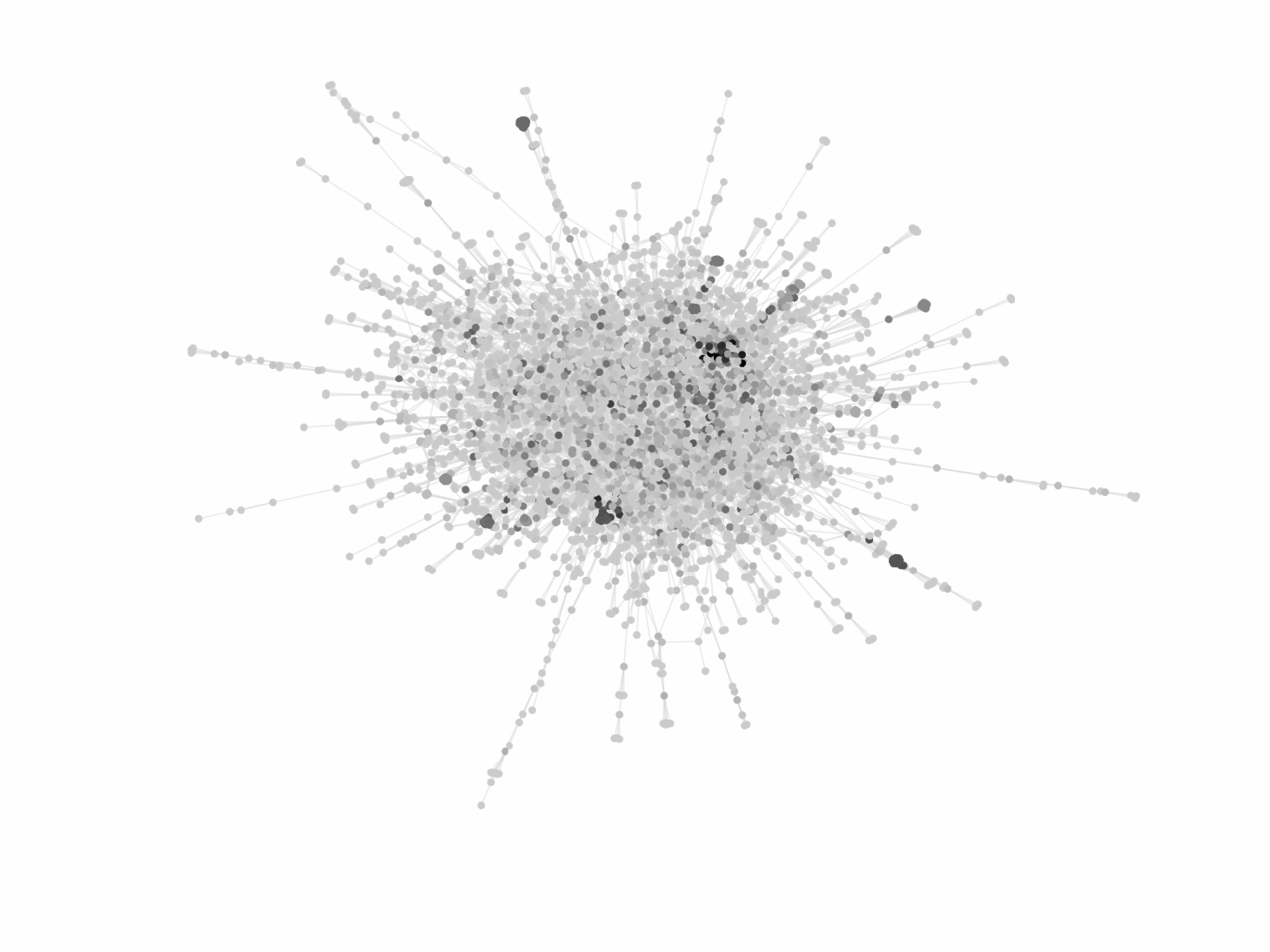}
		}
		\caption{\textbf{Large-Scale complex networks from the Stanford Network Analysis Project (SNAP).} (a)~The \texttt{ego-Facebook}, nodes are anonymized users from Facebook and edges indicate friendship status between them.  (b)~The \texttt{wiki-Vote} graph, each node represents a Wikipedia administrator and an directed edge represents a vote used for promoting a user to admin status.  (c)~The \texttt{ca-GrQc} graph is a collaboration network from arXiv authors with papers submitted to the General Relativity and Quantum Cosmology category, edges indicated co-authorship of a manuscript. The gray scale in the node colors shows the relative social power according to the left-eigenvector corresponding to the eigenvalue $1$.}
		\label{fig:graphs_real}
	\end{figure}

\begin{figure}[ht]
		\centering
		\subfigure{
	\includegraphics{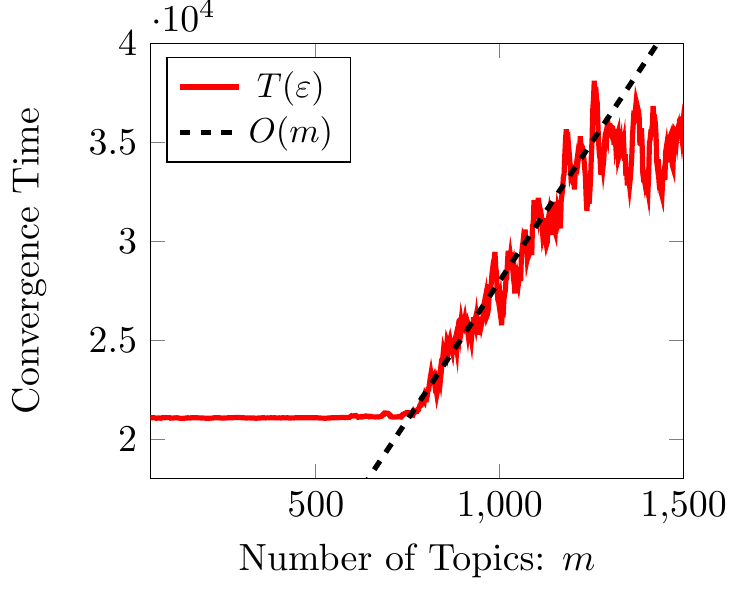}
}
\subfigure{
	\includegraphics{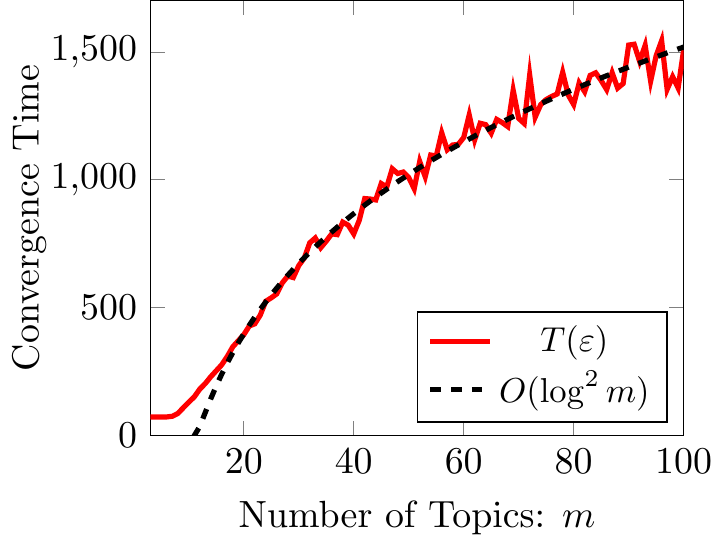}
}
\\
\subfigure{
	\includegraphics{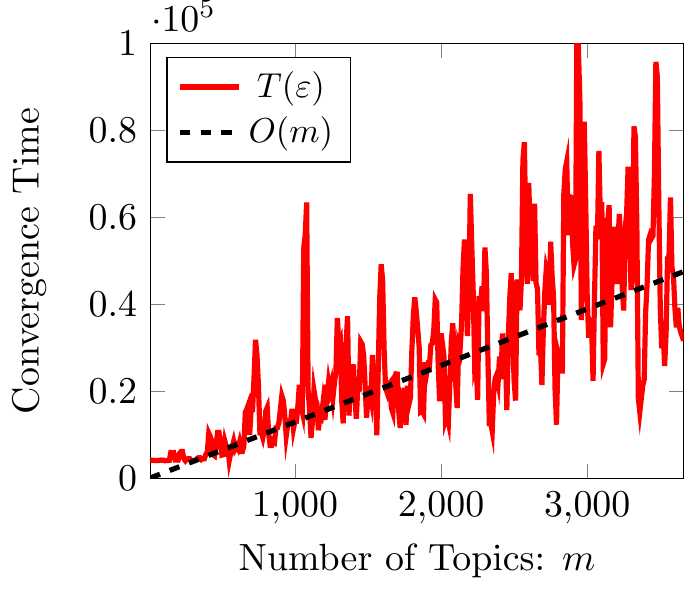}
}
		\caption{\textbf{Convergence time of a belief system over a large-scale complex network.} (a) The social network is the \texttt{ego-Facebook} graph and the logic constraints form a complete binary tree with an increasing number of topics. (b) The social network is the \texttt{wiki-Vote} graph and the logic constraints form Newman-Watts small-world graph with an increasing number of topics. (c) The social network is the \texttt{ca-GrQc} arXiv collaboration graph, and the logic constraints form an Erd\H{o}s-R\'enyi graph with an increasing number of topics.}
		\label{fig:mix_real}
	\end{figure}

\begin{figure}[ht]
		\centering
\includegraphics{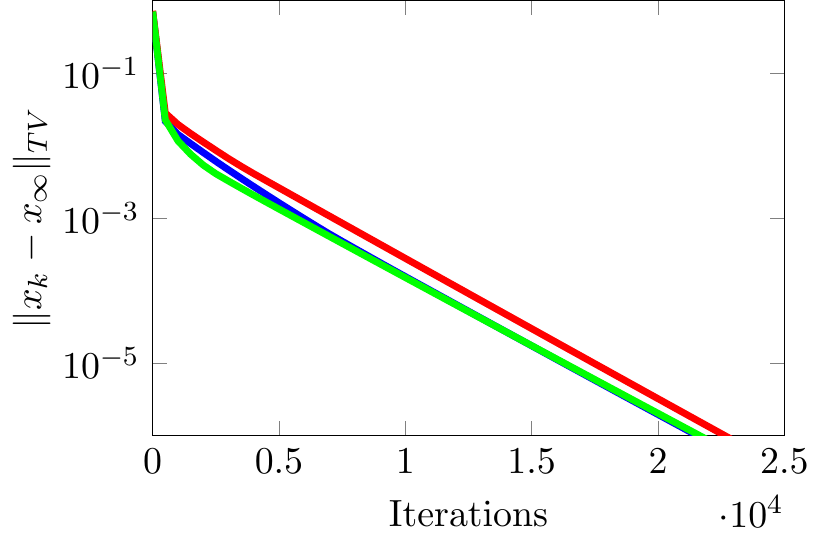}
		\caption{\textbf{Total variation distance between the beliefs and its limiting value as the number of iteration increases.} Results are shown for a particular subset of randomly selected agents.}
		\label{fig:ex4_distance}
\end{figure}

	\begin{figure}[ht]
		\centering
	\includegraphics{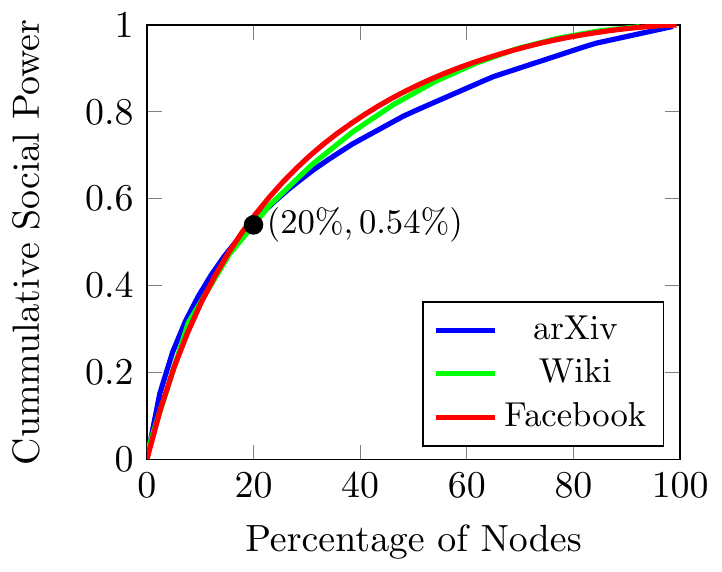}
		\caption{\textbf{Cumulative social power of the agents.} Each of the nodes in the graphs considered has some weight in the final value achieved by the belief system. In all three cases, the $20\%$ most important nodes account for $50\%$ of the final value. }
		\label{fig:social_power}
	\end{figure}

	%%%%%%%%%%%%%%%%%%%%%%%%%%%%%%%%%%%%%%%%%%%%%%%%%%%%%%%%%%%%%%%%%%%%%%
	%%%%%%%%%%%%%%%%%%%%%%%%%%%%%%%%%%%%%%%%%%%%%%%%%%%%%%%%%%%%%%%%%%%%%%
	%%%%%%%%%%%%%%%%%%%%%%%%%%%%%%%%%%%%%%%%%%%%%%%%%%%%%%%%%%%%%%%%%%%%%%
	
	\clearpage
	
	\bibliographystyle{naturemag}
	\bibliography{IEEEfull,all_refs}

\section*{Acknowledgements}

This research is supported partially by the National Science Foundation under grants no.~CPS 15-44953 and no.~CMMI-1463262, and by the Office of Naval Research under grant
no.~N00014-12-1-0998.

\section*{Author contributions statement}

AN, AO, and CAU conceived the project, derived the analytical results and wrote the manuscript. CAU performed the numerical simulations and analyzed the data. All authors reviewed the manuscript. 

\section*{Additional information}

The authors declare that they have no competing financial interests. Correspondence and requests for materials should be addressed to
CAU (cauribe@mit.edu).
	
	\newpage
	
	\section*{Supplementary Material}
	
	\renewcommand{\figurename}{\textbf{Supplementary Figure}}

	\section*{Supplementary Note 1: The Kronecker Product of Graphs}\label{sec_sup:kronecker}

	In this note, we define the Kronecker product of two matrices and the Kronecker product of two graphs. Also, we show some of the properties we will use in the proof of our main results regarding convergence, convergence time and limiting value of belief systems. 

\begin{definition}[\cite{wei62}]\label{def:kronecker}
	Let $A$ be a $m \times n$ matrix, and $C$ be a $p \times q$ matrix, the \textbf{Kronecker product} $A \otimes C$ is the $mp \times nq$ matrix defined as:
	\begin{align*}
	A\otimes C & = \left[ 
	\begin{array}{c c c }
	a_{11}C       &  \dots     & a_{1n}C \\
	\vdots             &  \ddots    & \vdots    \\
	a_{m1}C       &  \dots     & a_{mn}C \\
	\end{array}
	\right],
	\end{align*}
	or explicitly
	\begin{align*}
	A\otimes C & = \left[ 
	\begin{array}{c c c }
	a_{11}\left[ 
	\begin{array}{c c c }
	c_{11}       &  \dots     & c_{1q} \\
	\vdots             &  \ddots    & \vdots    \\
	c_{p1}       &  \dots     & c_{pq} \\
	\end{array}
	\right]       &  \dots     & a_{1n}\left[ 
	\begin{array}{c c c }
	c_{11}       &  \dots     & c_{1q} \\
	\vdots             &  \ddots    & \vdots    \\
	c_{p1}       &  \dots     & c_{pq} \\
	\end{array}
	\right]  \\
	\vdots             &  \ddots    & \vdots    \\
	a_{m1}\left[ 
	\begin{array}{c c c }
	c_{11}       &  \dots     & c_{1q} \\
	\vdots             &  \ddots    & \vdots    \\
	c_{p1}       &  \dots     & c_{pq} \\
	\end{array}
	\right]         &  \dots     & a_{mn}\left[ 
	\begin{array}{c c c }
	c_{11}       &  \dots     & c_{1q} \\
	\vdots             &  \ddots    & \vdots    \\
	c_{p1}       &  \dots     & c_{pq} \\
	\end{array}
	\right]  \\
	\end{array}
	\right]\\
	& = \left[ 
	\begin{array}{c c c c c c c c c}
	a_{11}c_{11}       &  \dots     & a_{11}c_{1q}  & \dots & a_{1n}c_{11}       &  \dots     & a_{1n}c_{1q} \\
	\vdots       &  \ddots     & \vdots   &  &\vdots         &  \ddots     & \vdots  \\
	a_{11}c_{p1}       &  \dots     & a_{11}c_{pq}  & \dots & a_{1n}c_{p1}       &  \dots     & a_{1n}c_{pq} \\
	\vdots       &       & \vdots   &  &\vdots         &     & \vdots  \\
	\vdots       &       & \vdots   &  &\vdots         &     & \vdots  \\
	a_{m1}c_{11}       &  \dots     & a_{m1}c_{1q}  & \dots & a_{mn}c_{11}       &  \dots     & a_{mn}c_{1q} \\
	\vdots       &  \ddots     & \vdots   &  &\vdots         &  \ddots     & \vdots  \\
	a_{m1}c_{p1}       &  \dots     & a_{m1}c_{pq}  & \dots & a_{mn}c_{p1}       &  \dots     & a_{mn}c_{pq} \\
	\end{array}
	\right]  .
	\end{align*}
\end{definition}

Next, we will enumerate some useful properties of the Kronecker product.

\begin{enumerate}
	\item Bilinearity and associativity: for matrices $A$, $B$ and $C$, and a scalar $k$, it holds:
	\begin{align*}
	A \otimes (B + C) &= A\otimes B + A \otimes C\\
	(A + B)\times C & = A \otimes C + B \otimes C \\
	(kA) \otimes C & = A \otimes (kB) = k(A\otimes B) \\
	(A\otimes B) \otimes C & = A\otimes (B \otimes C).
	\end{align*} 
	\item Non-Commutative: In general $A \otimes B \neq B \otimes A$. However, there exist commutation matrices $P$ and $Q$ such that:
	\begin{align*}
	A \otimes B &  = P(B \otimes A) Q,
	\end{align*}
	and if $A$ and $B$ are square matrices then $P = Q'$.
	\item Mixed-product property: for matrices $A$, $B$, $C$ and $D$:
	\begin{align*}
	(A\otimes B)(C\otimes D) &  = (AC)\otimes (BD).
	\end{align*}
\end{enumerate}

Next, we introduce the Kronecker product of graphs and some of its properties.

\begin{definition}[ {\cite[Definition~1]{wei62}}]\label{def:prod_graph}
	The Kronecker  (also known as categorical, direct, cardinal, relational, tensor, weak direct or conjunction) product $\mathcal{G} = \mathcal{G}_1 \otimes \mathcal{G}_2$ of two graphs $\mathcal{G}_1 = (V_1,E_1)$ and $\mathcal{G}_2 = (V_1,E_1)$ is a graph $\mathcal{G} = (V,E)$ where $V = V_1 \times V_2$; and $(u,u') \to (v,v') \in E$ if and only if $u \to v \in E_1$ and $u' \to v' \in E_2$. Moreover, the adjacency matrix of the graph $\mathcal{G}$ is the Kronecker product of the adjacency matrices of $\mathcal{G}_1$ and $\mathcal{G}_2$.
\end{definition}

%	\begin{theorem}[Theorem 5.29 in Imrich and Klavzar~\cite{imr00}]\label{thm:kron}
%		Let $\mathcal{G}$ and $\mathcal{H}$ be graphs with at least one edge. Then $\mathcal{G} \otimes \mathcal{H}$ is connected if and only if both $\mathcal{G}$ and  $\mathcal{H}$ are connected, and at least one of them is nonbipartite. Furthermore, if both  $\mathcal{G}$ and  $\mathcal{H}$ are connected and bipartite then  $\mathcal{G} \otimes \mathcal{H}$ has exactly two connected components.
%	\end{theorem}

\begin{theorem}[{\cite[Theorem~1]{McAndrew1963}} ]\label{thm:kron_directed}
	\setcounter{theorem}{0}
	Let $\mathcal{G}$ and $\mathcal{H}$ be strongly connected graphs. Let $d_1 =d(\mathcal{G})$, $d_2 =d(\mathcal{H})$, $d_3 =\text{gcd}(d_1,d_2)$ and $D = \text{lcm}(d1,d2)$. Then, the number of components in $\mathcal{G} \otimes \mathcal{H}$ is $d_3$. Moreover, for any component $\mathcal{B}$ of  $\mathcal{G} \otimes \mathcal{H}$, $d(\mathcal{B}) = D$.
\end{theorem}

%	\section*{Supplementary Note 2: Main Technical Results}
%	
%	\input{sup_main.tex}
	
		\begin{figure}
		\centering
		\subfigure[]{
			\includegraphics{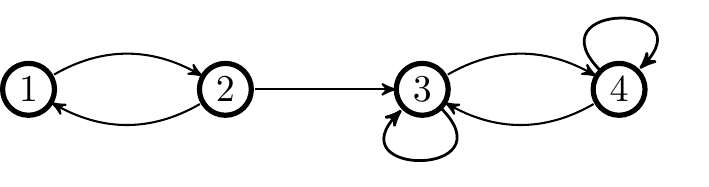}
		}
		\subfigure[]{
			\includegraphics{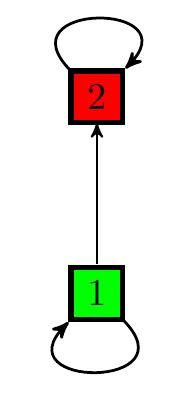}
		}
		\subfigure[]{
			\includegraphics{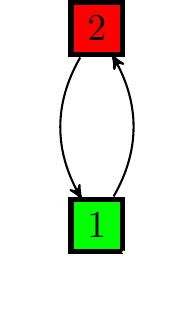}
		}
		\subfigure[]{
			\includegraphics{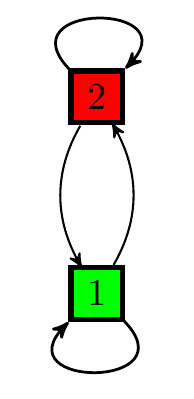}
		}
		\\
		\subfigure[]{
			\includegraphics{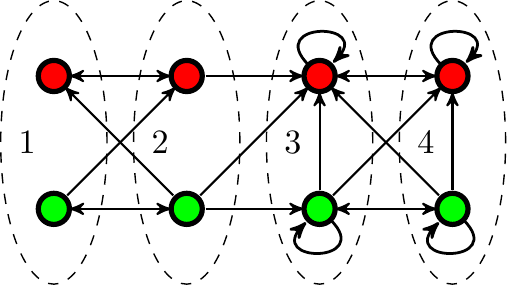}
		}
		\subfigure[]{
			\includegraphics{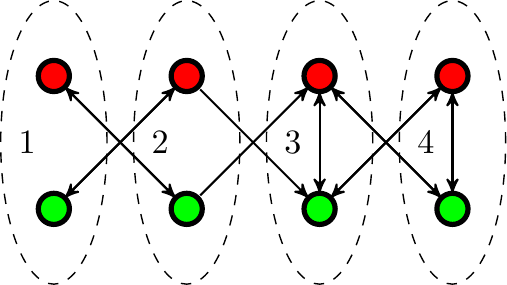}
		}
		\subfigure[]{
			\includegraphics{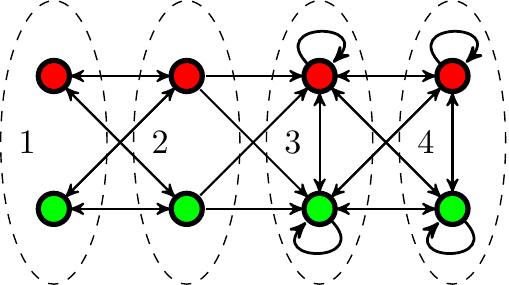}
		}
		\caption{\textbf{The influence of the logic constraints in the resulting aggregated belief system.} (a)~The network of agents, where agent $1$ follows the opinion of agent $2$, agent $2$ is influenced by agent $1$ and $3$, agent $3$ is influenced by its own opinion, and the opinion of agent $4$ and agent $4$ is influenced by agent $3$ as well as its own. (b)~The opinion on statement $1$ is influenced by the belief on statement $2$. (c)~The opinion on statements $2$ and $1$ follow each other. (d)~The opinion on statements $2$ and $1$ influence each other (e-g) The belief systems with the network of agents in (a) and  logic constraints in (b-d).}
		\label{learning_network}
	\end{figure}

\begin{figure}
		\centering
		\subfigure[]{
			\includegraphics{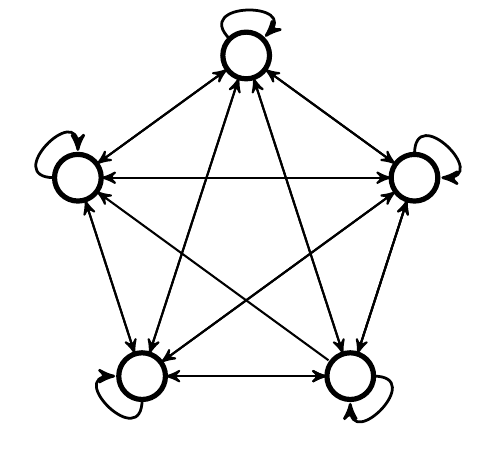}
		}
		\subfigure[]{
			\includegraphics{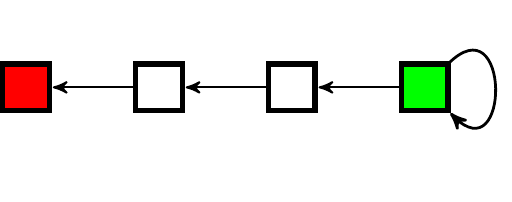}
		}
		\\
		\subfigure[]{
			\includegraphics{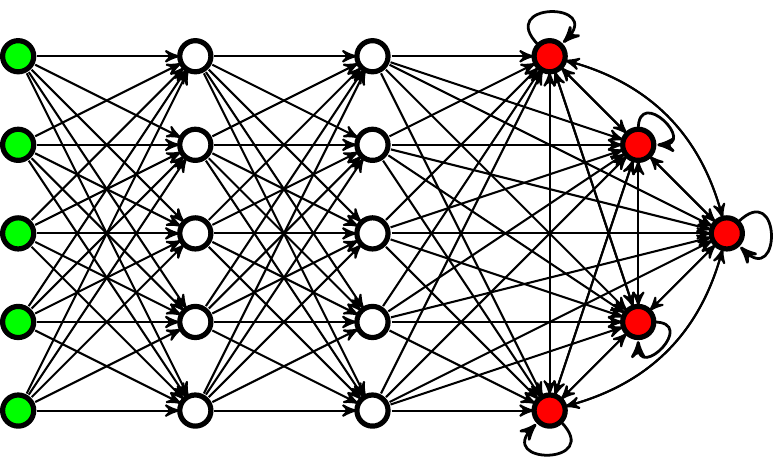}
		}
		\caption{\textbf{An additional example of a belief system.} A product of a complete graph/cycle graph with $5$ nodes and a path graph of $4$ logical belief constraints. (a)~A complete graph with $5$ agents. (b)~A directed path graph with $5$ nodes. (c) The belief system graph from the network of agents in (a) and the network of logic constraints in (b).}
		\label{fig:supp4}
	\end{figure}

\begin{figure}
		\centering
		\subfigure[]{
			\includegraphics{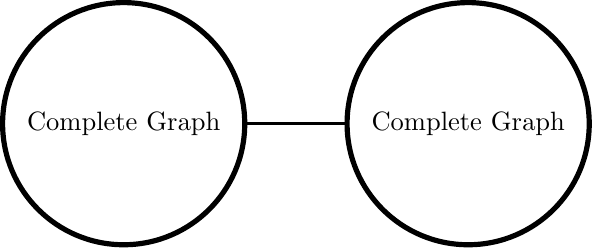}
		}
		\subfigure[]{
			\includegraphics{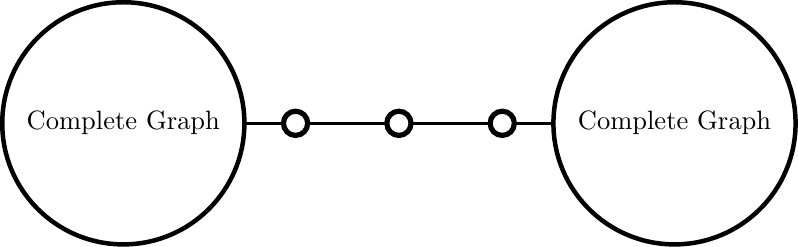}
		}
		\\
		\subfigure[]{
			\includegraphics{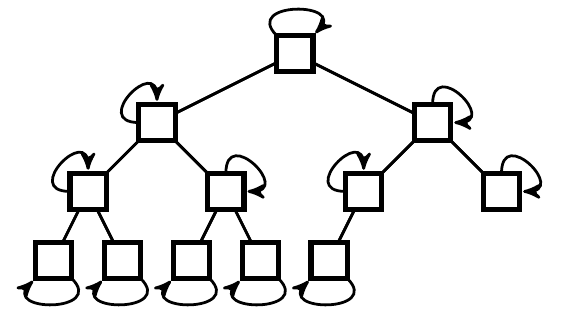}
		}
		\subfigure[]{
			\includegraphics{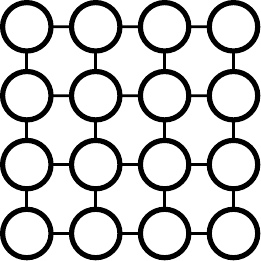}
		}
		\\
		\subfigure[]{
			\includegraphics{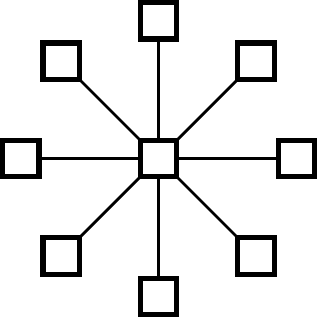}
		}
		\subfigure[]{
			\includegraphics{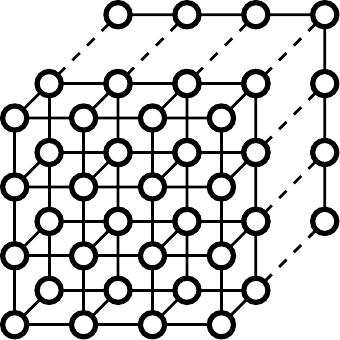}
		}
		\subfigure[]{
			\includegraphics{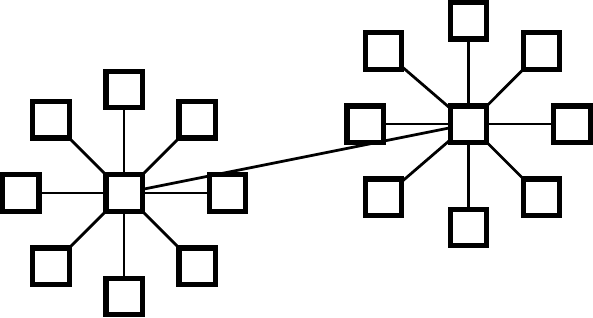}
		}
		\caption{\textbf{Examples of common graph families.} (a) Dumbbell graph, two complete graphs connected by an edge. (b) Bolas graph, two complete graphs connected by a path. (c) Complete binary tree. (d) $2$-d grid or lattice. (e) Star graph. (f) $3$-d grid. (g) Two-star graph connected to their centers.}
		\label{fig:ex_torus3_star2}
	\end{figure}

\begin{figure}
		\centering
		\subfigure[]{
			\includegraphics[width=0.3\textwidth]{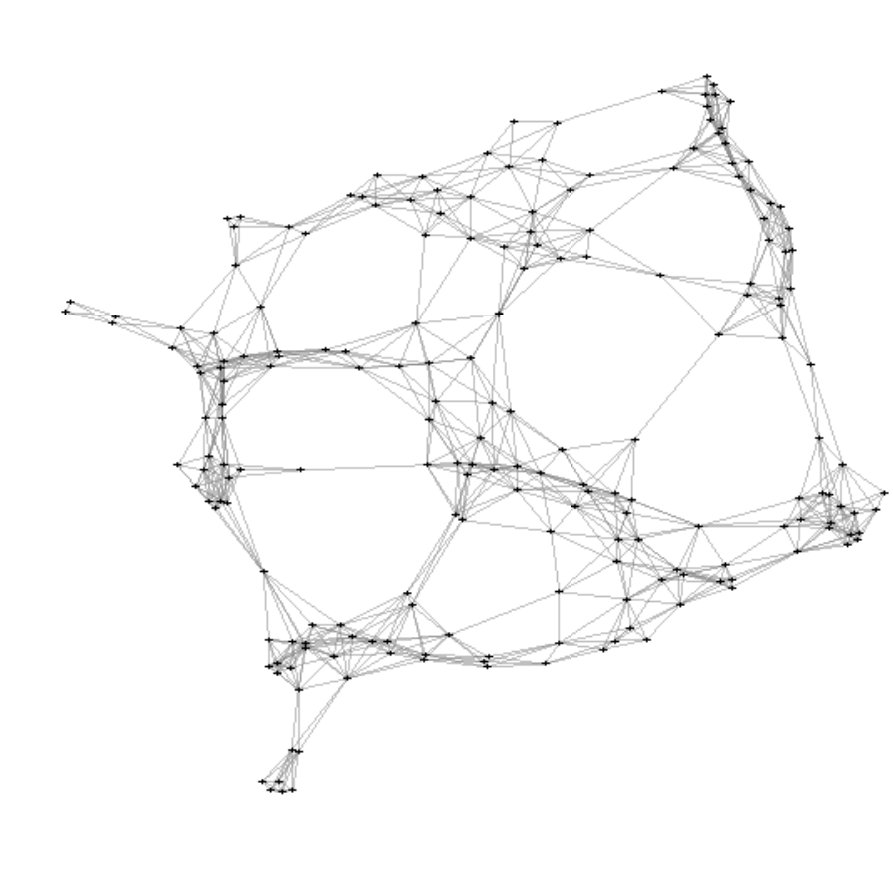}
		}
		\subfigure[]{
			\includegraphics[width=0.3\textwidth]{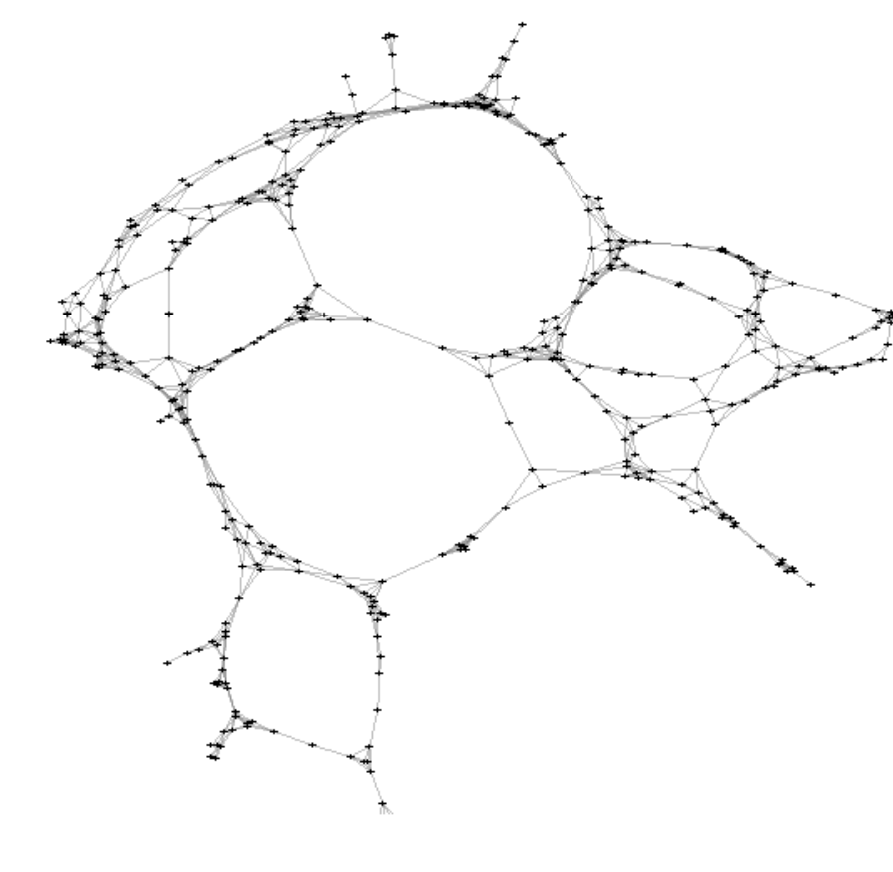}
		}
		\subfigure[]{
			\includegraphics[width=0.3\textwidth]{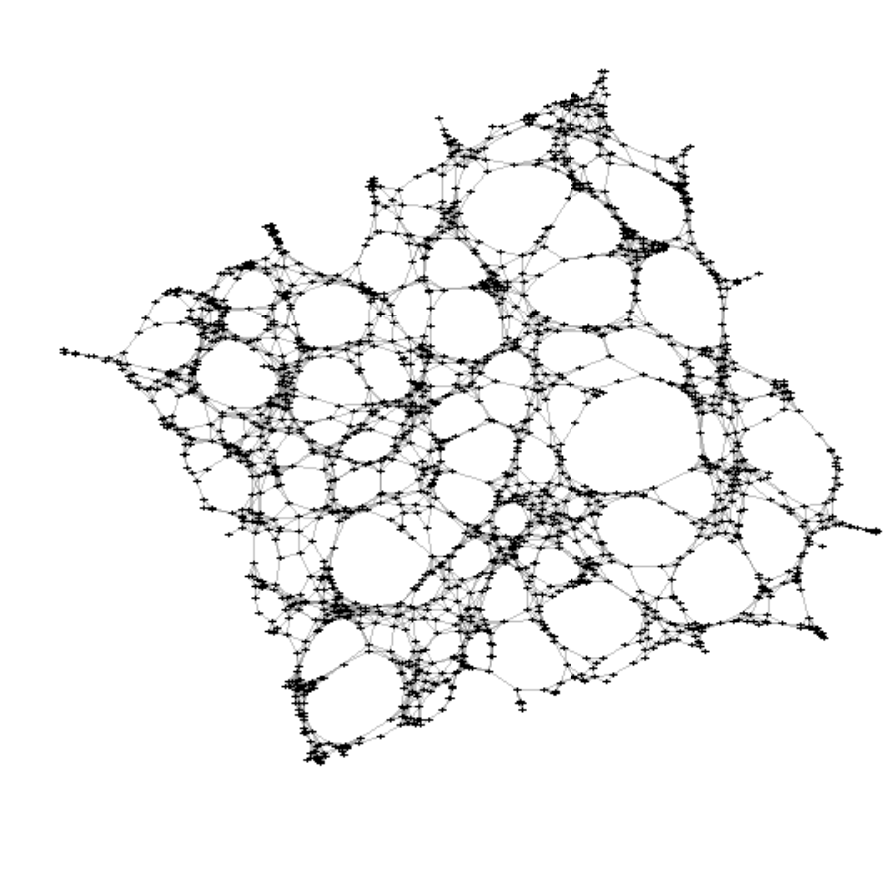}
		}
		\\    
		\subfigure[]{
			\includegraphics[width=0.3\textwidth]{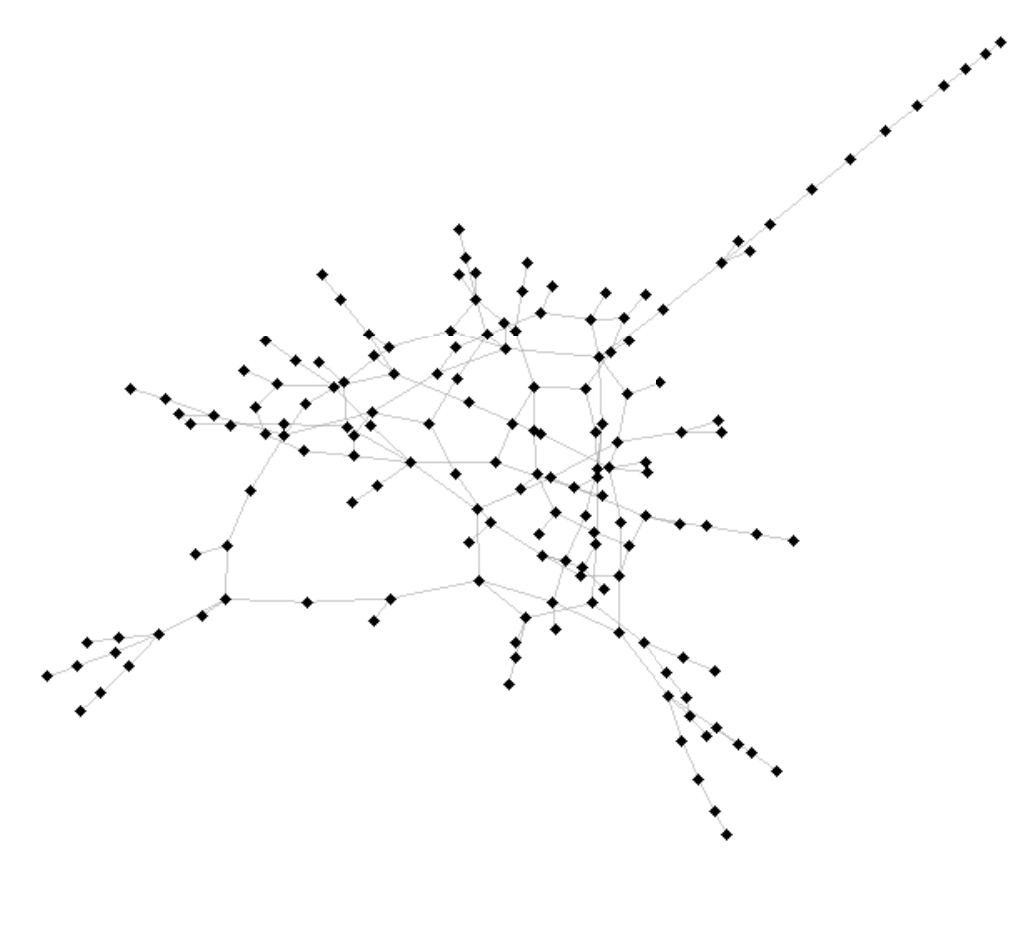}
		}
		\subfigure[]{
			\includegraphics[width=0.3\textwidth]{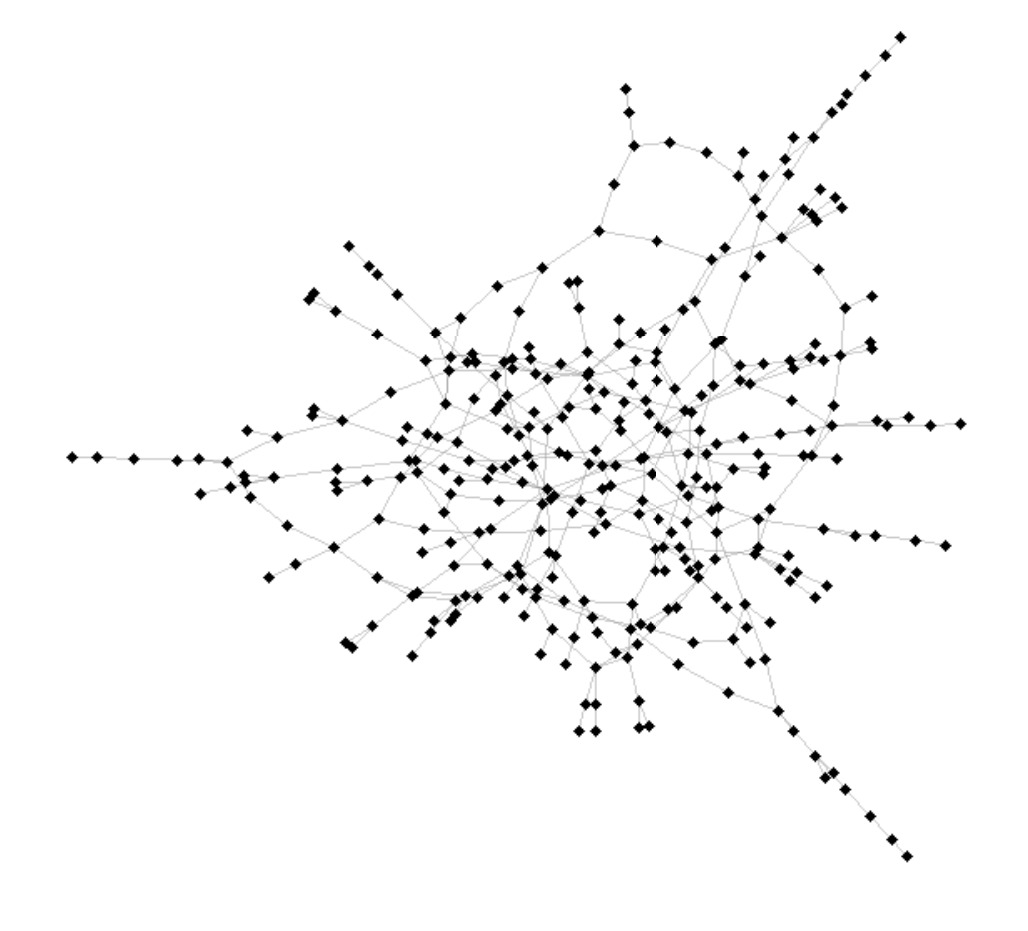}
		}
		\subfigure[]{
			\includegraphics[width=0.3\textwidth]{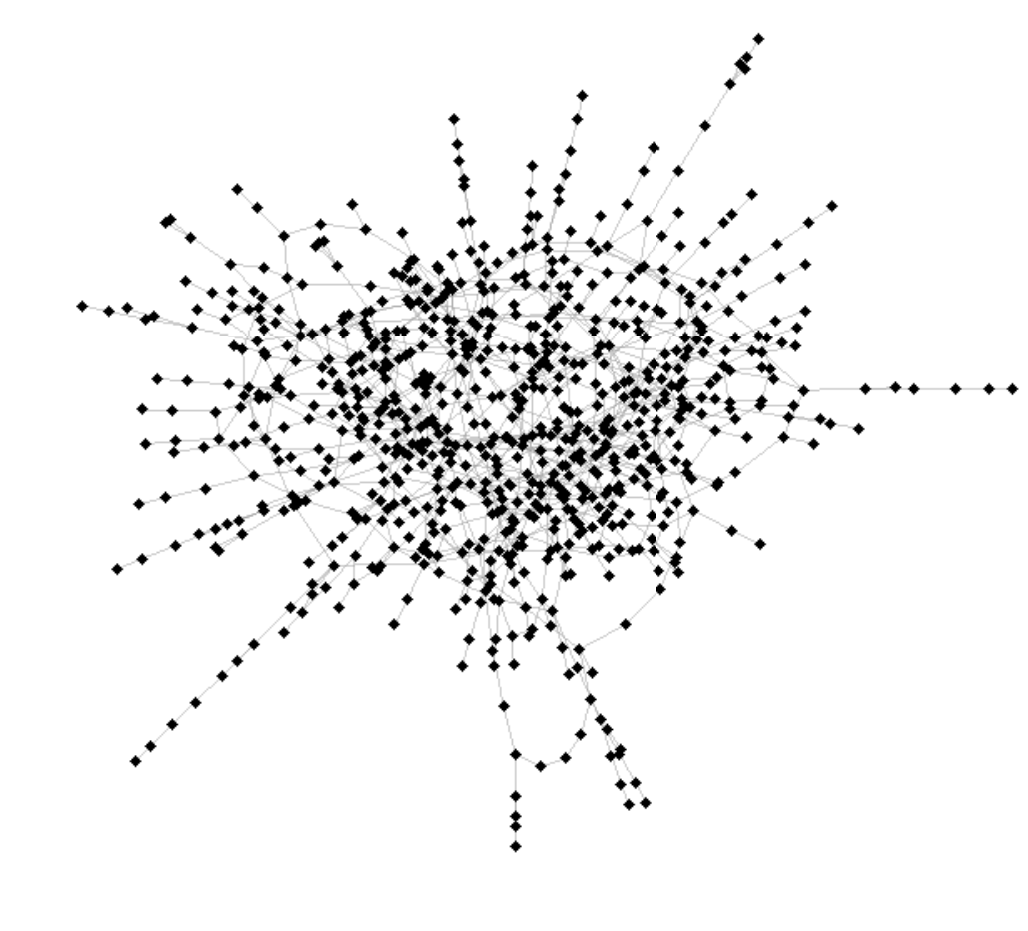}
		}
		\\
		\subfigure[]{
			\includegraphics[width=0.3\textwidth]{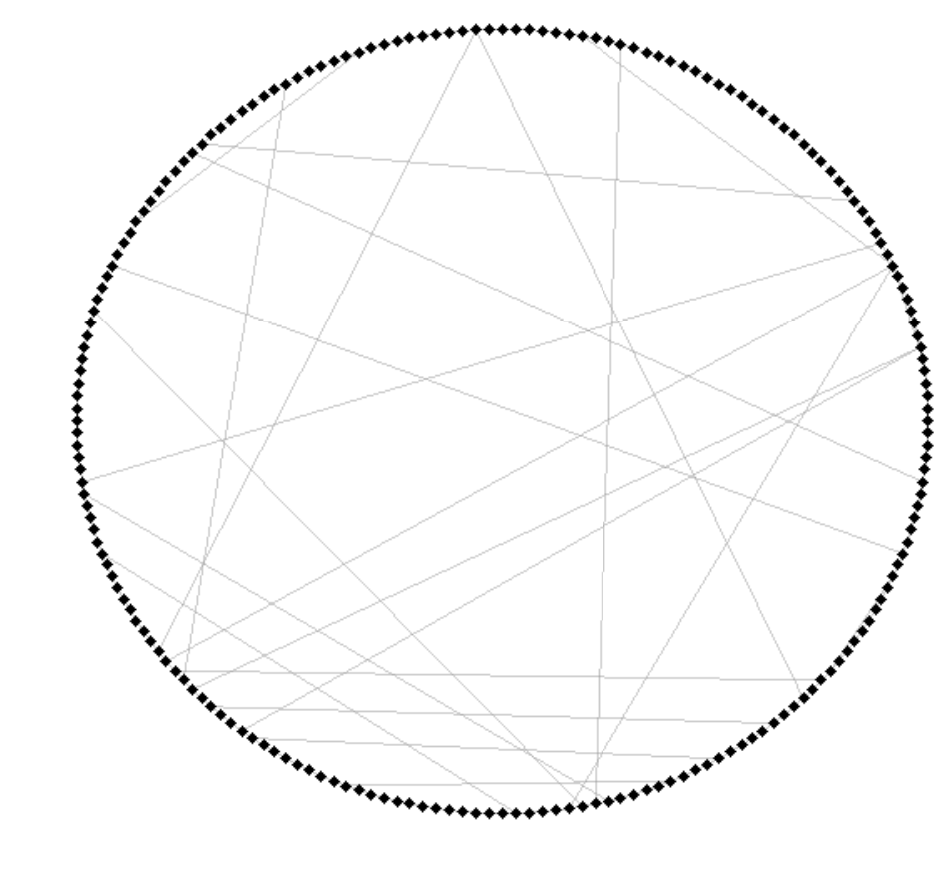}
		}
		\subfigure[]{
			\includegraphics[width=0.3\textwidth]{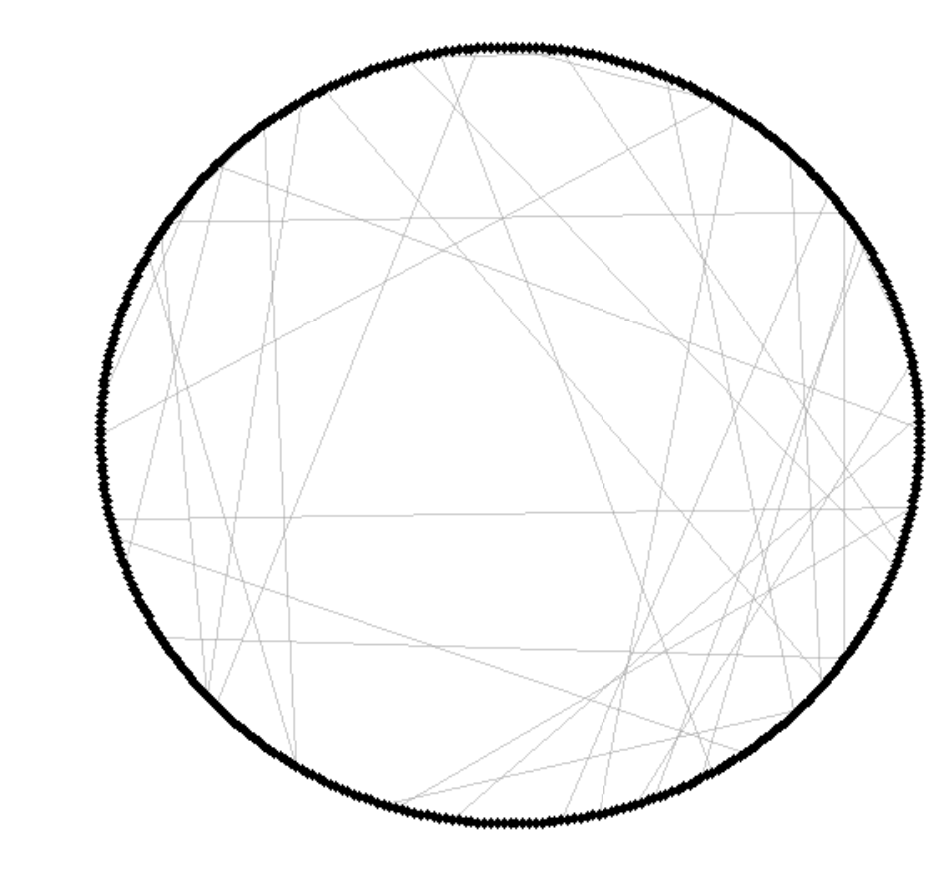}
		}
		\subfigure[]{
			\includegraphics[width=0.3\textwidth]{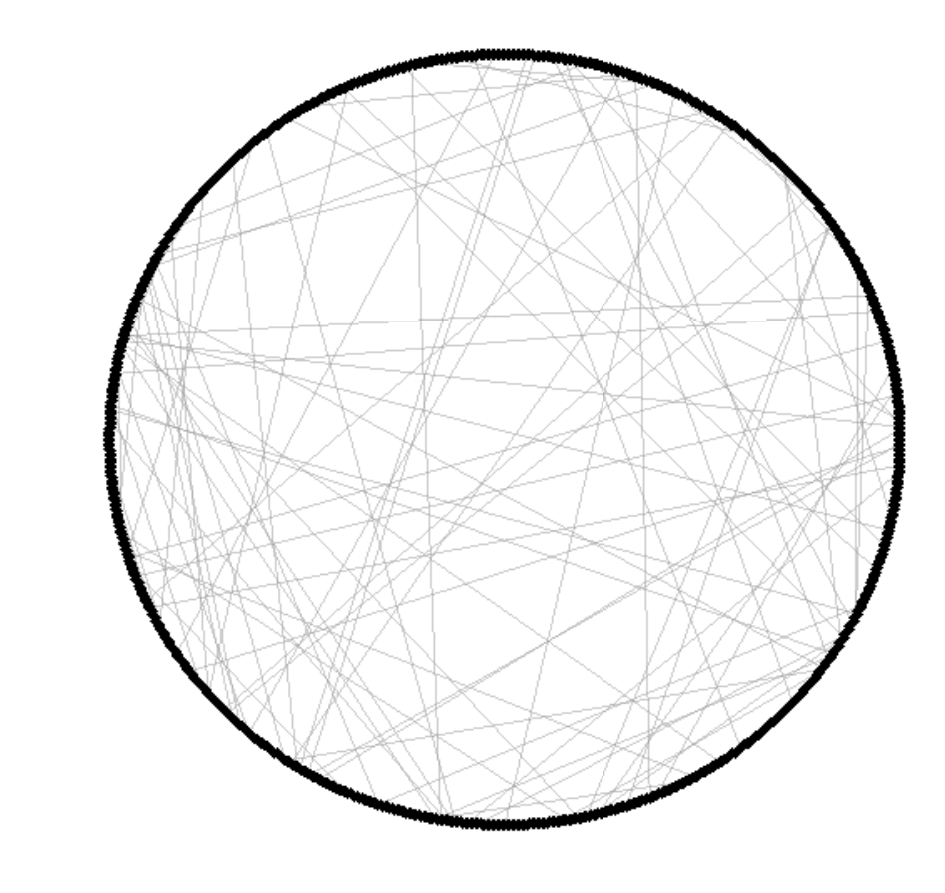}
		}
		\caption{\textbf{Examples of random graphs.} (a-c)~Geometric random graphs with $200$, $400$ and $2000$ nodes respectively. A geometric random graph is a result of randomly placing $n$ nodes in a metric space and adding an edge between two nodes if and only if their distance is smaller than certain radius $r$ \cite{pen03}.
			\mbox{(d-f)}~\mbox{\protect Erd\H{o}s-R\'enyi} random graphs with $200$, $400$ and $1000$ nodes respectively. An $\mathcal{G}_{n,p}$ Erd\H{o}s-R\'enyi graph is the result of adding edges independently with probability $p$ to a set of $n$ nodes \cite{erd60}. (g-i)~Newman-Watts Random Graphs with $200$, $400$ and $1000$ nodes respectively. The Newman-Watts graph $H_{n,k,p}$ is the random graph obtained from a $(n,k)$-ring graph by independently adding edges with probability $p$ \cite{new99}.
		}
		\label{fig:random_graphs}
	\end{figure}

\begin{figure}
		\centering
	\subfigure[]{
	\includegraphics{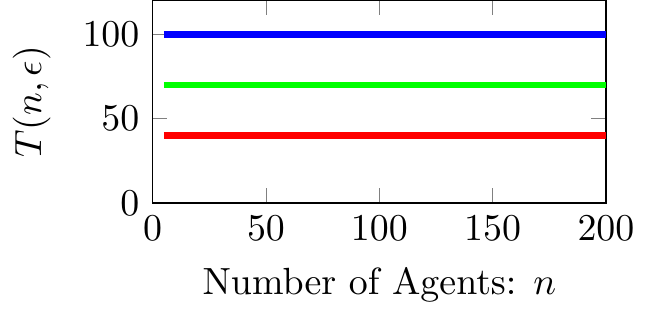}
}
\subfigure[]{
	\includegraphics{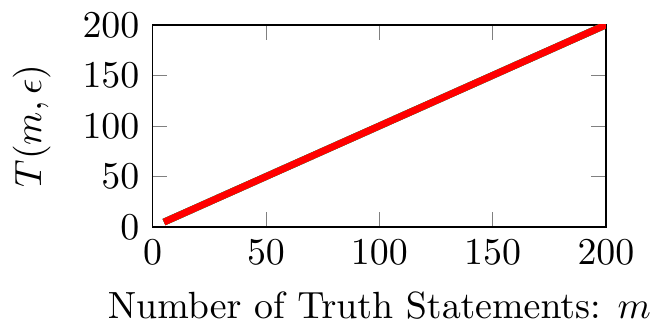}
}
\\
\subfigure[]{
	\includegraphics{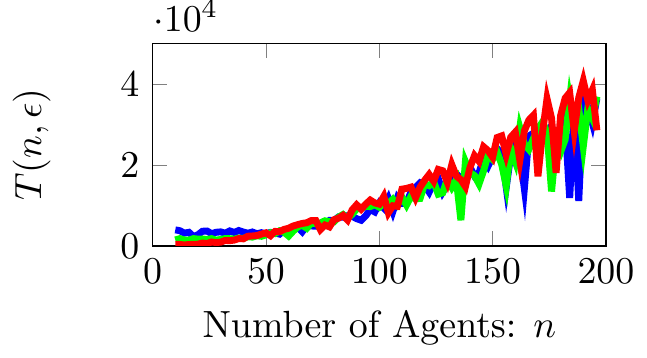}
}
\subfigure[]{
	\includegraphics{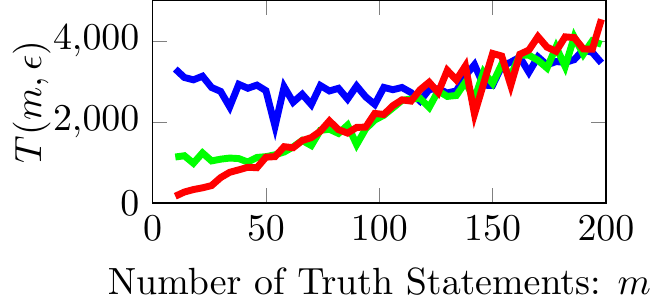}
} 
\\
\subfigure[]{
	\includegraphics{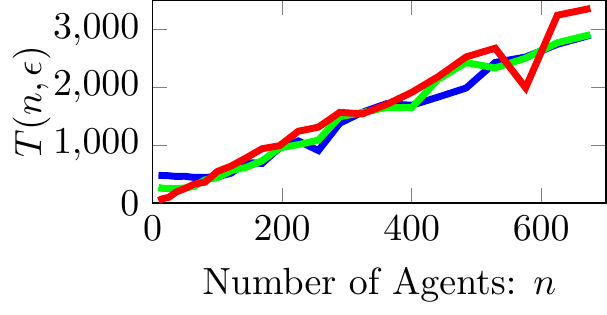}
}
\subfigure[]{
	\includegraphics{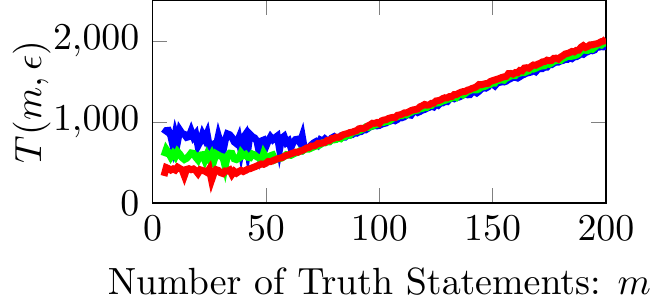}
}
\\
\subfigure[]{
	\includegraphics{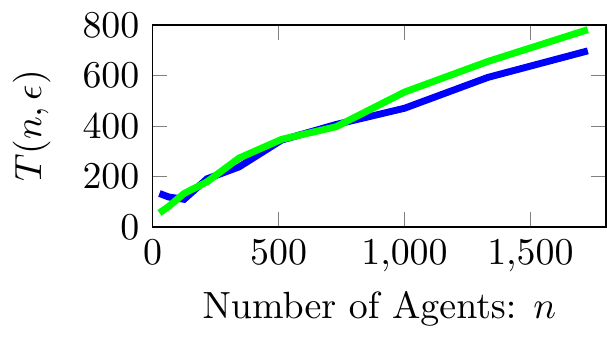}
}
\subfigure[]{
	\includegraphics{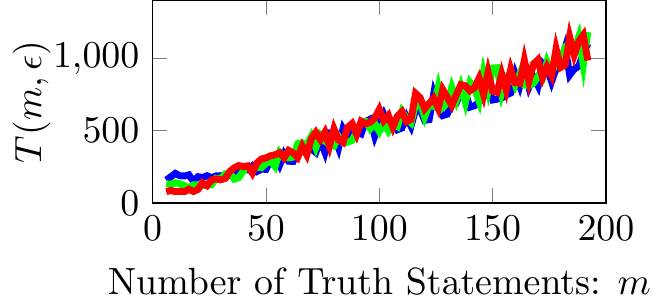}
} 
		\caption{\textbf{Convergence time for different examples of networks of agents and network of truth statements in a belief system.} Varying the number of agents for a: (a)~complete graph, (c)~dumbbell graph, (e)~$2$-d grid and (g)~$3$-d grid. Varying the number of truth statements for a: (b)~directed path, (d)~complete binary tree, (f)~star graph and (h)~two joined star graphs.}
		\label{fig:complexities}
	\end{figure}

\begin{figure}
		\centering
		\subfigure[]{
	\includegraphics{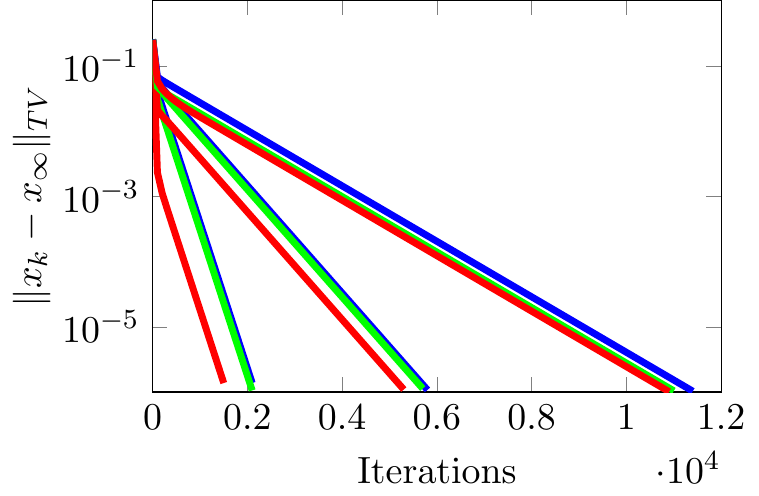}
}
\subfigure[]{
	\includegraphics{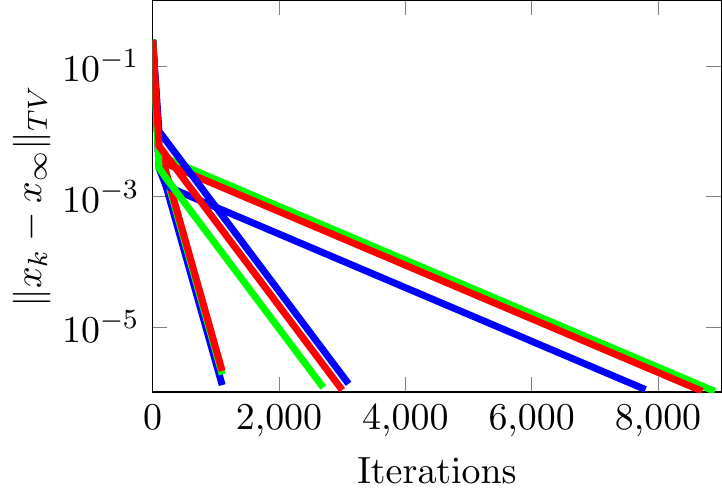}
}
\\
\subfigure[]{
	\includegraphics{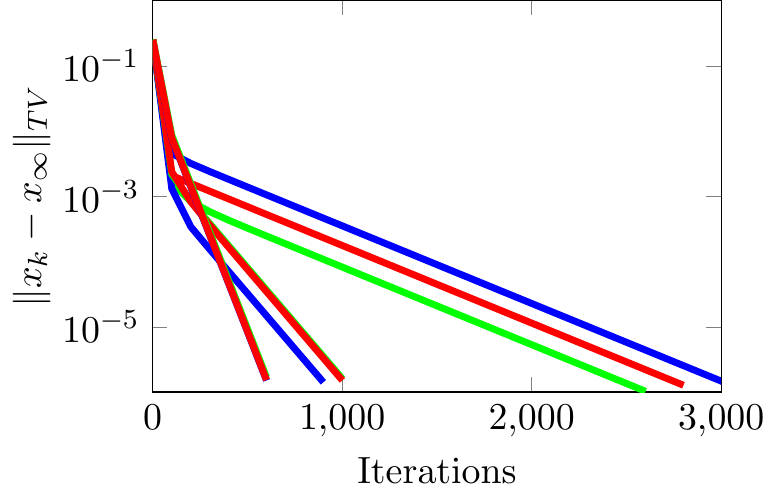}
}
\subfigure[]{
	\includegraphics{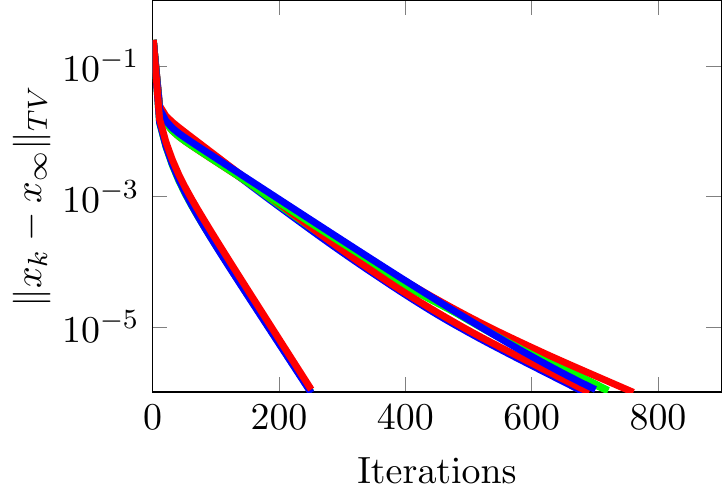}
}
		\caption{\textbf{Linear convergence of the belief system.} Distance to the final value of a belief system with: (a)~a directed cycle network of agents and a directed path of truth statements, (b)~a dumbbell network of agents and a complete binary tree of truth statements, (c)~a $2$-d grid of agents and a star network of truth statements, (d)~a $3$-d grid of agents and a two-jointed star network of truth statements.}
		\label{fig:convergence}
	\end{figure}

\begin{figure}
		\centering
		\subfigure[]{
	\includegraphics{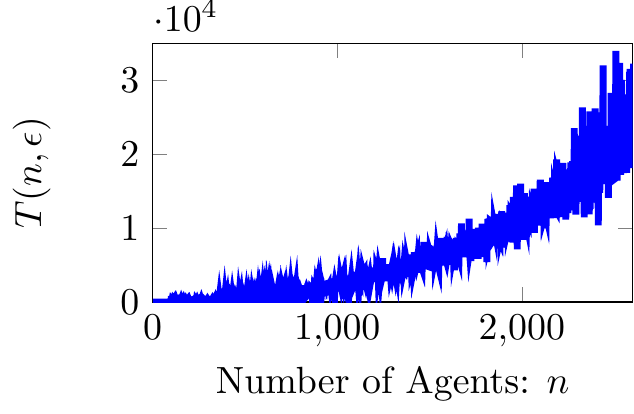}
}
\subfigure[]{
	\includegraphics{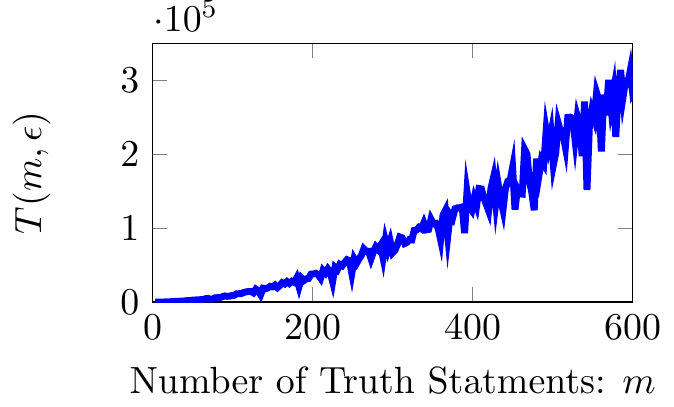}
}
\\
\subfigure[]{
	\includegraphics{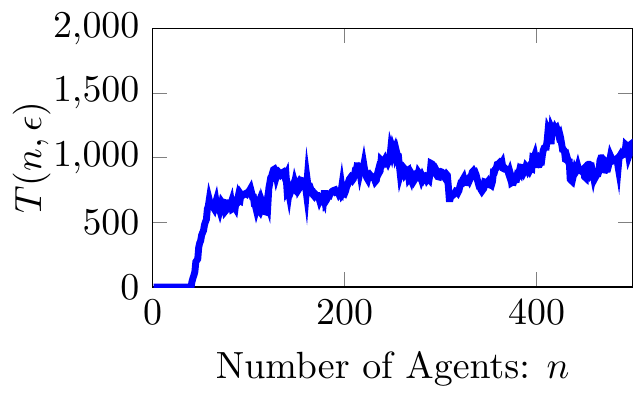}
}
\subfigure[]{
	\includegraphics{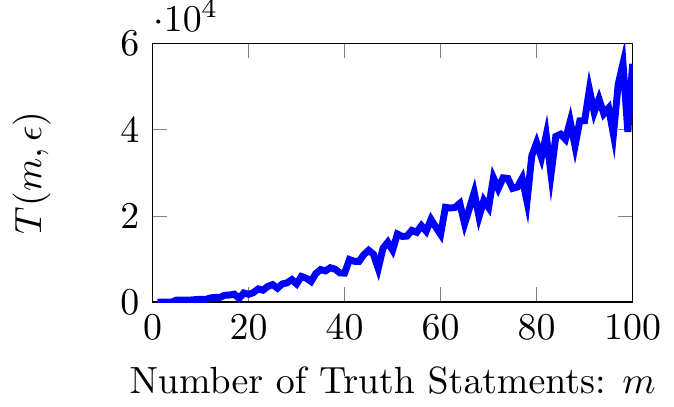}
}
\\
\subfigure[]{
	\includegraphics{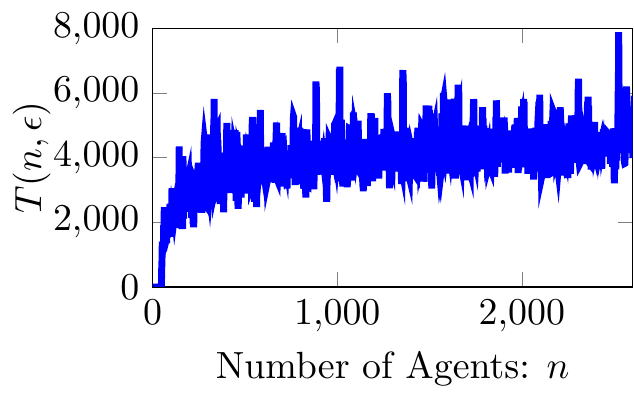}
}
\subfigure[]{
	\includegraphics{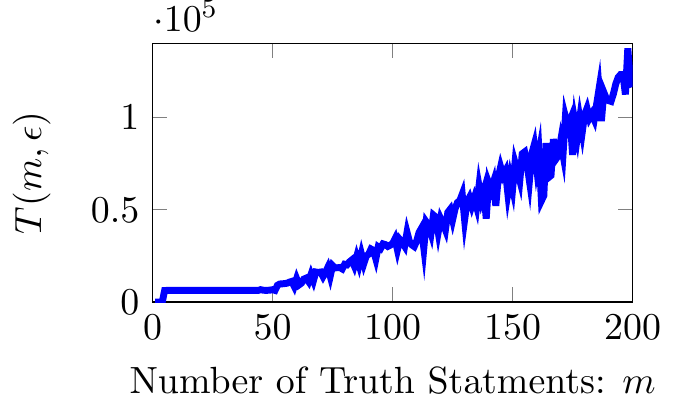}
}
		\caption{\textbf{Convergence time dependency for random graphs.} (a)~Varying the number of agents in a geometric random graph with a fixed number of truth statements in a Bolas graph. (b)~Varying the number of truth statements in a Bolas graph with a fixed number of agents in a geometric random graph. (c)~Varying the number of agents in an \mbox{\protect Erd\H{o}s-R\'enyi} random graph with a fixed number of truth statements in a dumbbell graph. (d)~Varying the number of truth statements in a dumbbell graph with a fixed number of agents in an \mbox{\protect Erd\H{o}s-R\'enyi} random graph. (e)~Varying the number of agents in a \mbox{\protect Newman-Watts} random graph with a fixed number of truth statements in an undirected path graph. (f)~Varying the number of truth statements in an undirected path graph with a fixed number of agents in a \mbox{\protect Newman-Watts} random graph random graph.}
		\label{fig:rand_complexities}
	\end{figure}

\begin{figure}
		\centering
		\subfigure[]{
	\includegraphics{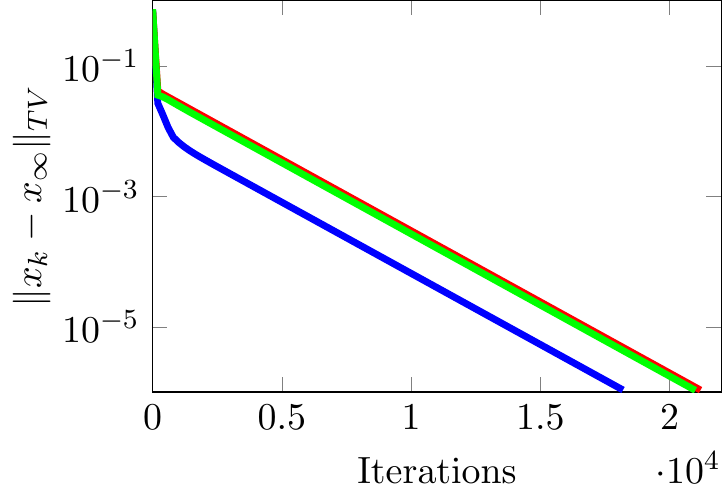}
}
\subfigure[]{
	\includegraphics{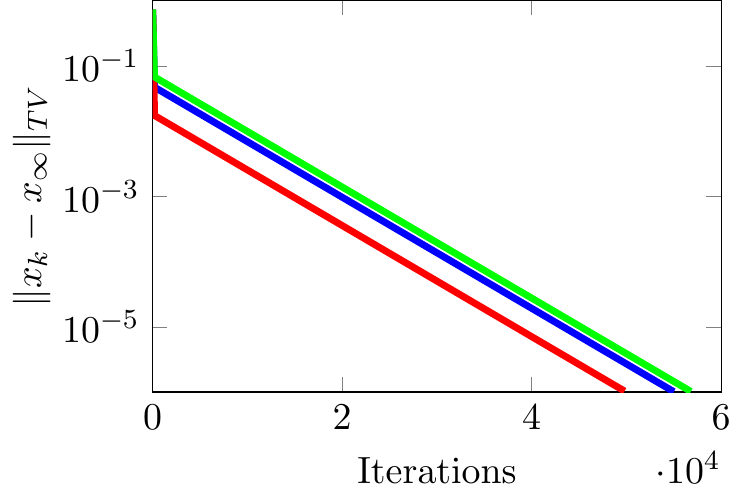}
}
\\
\subfigure[]{
	\includegraphics{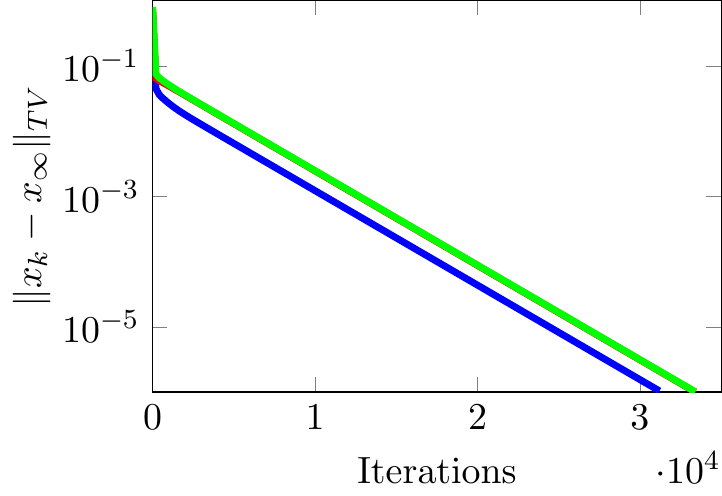}
}
		\caption{\textbf{Linear convergence rate of the belief system with random networks of agents.} (a)~Distance to the stationary distribution for a network of $200$ agents modeled as a geometric random graph and a network of $150$ truth statements modeled as a Bolas graph. (b)~Distance to the stationary distribution for a network of $500$ agents modeled as an \mbox{\protect Erd\H{o}s-R\'enyi} random graph and a network of $100$ truth statements modeled as a dumbbell graph. (c)~Distance to the stationary distribution for a network of $500$ agents modeled as a small-world random graph and a network of $100$ truth statements modeled as an undirected path graph.}
		\label{fig:rand_conv}
	\end{figure}

\bgroup
	\def\arraystretch{1.2}
	\begin{table}
		\centering
		\begin{tabular}{|cc|}  
			\hline
			\textbf{Network Topology } &  \textbf{Mixing Time } \\
			\hline
			%{Complete}                          & $ O(\log (1/\epsilon))$                         \\
			{Cycle}   \cite[Section~$5.3.1$]{lev09}                            & $ O(n^2 )$                      \\
			{Path}      \cite{ike09,beveridge2013exact}                       & $ O(n^2 )$                        \\
			{Star Graph}   \cite{beveridge2013exact}                     & $ O(1 )$     \\
			{Two Joined Star Graphs}            & $ O(1)$       \\
			{Dumbbell Graph}  \cite{kan06}                 & $ O(n^2 )$                        \\
			{Lollipop}      \cite{ald02}                & $ O(n^2 )$  \\
			%			\cu{Barbell}          \cite{ald02}                  & $ O(n^3 )$  \\
			{Bolas Graph}   \cite{ald02}                    & $ O( n^3  )$       \\
			{Complete Binary Tree} \cite[Section~$5.3.4$]{lev09}          & $ O(n )$                        \\
			%$k$-d Cube with Loops    \cite{lov93}        & $ O((1-1/k) )$        \\
			{$k$-d Hypercube $\{0,1\}^k$} \cite[Section~$5.3.3$]{lev09}       & $ O(k\log k+ k )$        \\
			%Lovasz Graph $\mathcal{C}_n^k$ \cite{lov93} & $ O((1-1/(kn^2)) )$        \\
			{L-Lattice on $\mathcal{Z}_n \times \mathcal{Z}_n$}   \cite{avi05,cha96}                     & $ O(n^2)$     \\
			{$2$-d Grid}   \cite{avi05,cha96}                     & $ O(n (\log n+ ))$     \\
			
			{$3$-d Grid}  \cite{avi05,cha96}                      & $ O(n^{2/3}(\log n+ ))$       \\
			
			{$k$-d Grid} \cite{avi05,cha96}                       & $ O(2k^2n^{2 / k} (\log n+ ))$  \\
			{$2$-d Torus}      \cite[Section~$5.3.3$]{lev09}                  & $ O(n^2 )$  \\
			{$3$-d Torus}      \cite[Section~$5.3.3$]{lev09}                   & $ O(n^2 )$  \\
			{$k$-d Torus}     \cite[Section~$5.3.3$]{lev09}                   & $ O(k^2n^2 )$  \\
			
			{Eulerian Graph}      \cite{mon09}                  & $ O(|E|^2 )$  \\
			Lazy Eulerian with degree $d$-degree          \cite{boc16}                  & $ O(n|E| )$  \\
			{Eulerian: $d$-degree, max-degree weights and expansion}          \cite{mon09}                  & $ O(n^2d )$  \\
			%	Lamplighter on $k$-Hypercube  \cite{kom14}  & $ O( k 2^k  )$       \\
			%	Lamplighter on $(k,n)$-Torus   \cite{kom14} & $ O( k n ^ k  )$       \\
			Geometric Random Graph: $\mathcal{G}^d(n,r)$      \cite{boy05}   & $ O( r^{-2}\log n  )$       \\
			Geometric Random Graph: $\mathcal{G}^2(n,\Omega(\text{polylog}(n)))$  \cite{avi07}       & $ O(\text{polylog}(n)  )$       \\
			Erd\H{o}s-R\'enyi: $\mathcal{G}(n,c/n)$, $c>1$          \cite{ben14,fou07}   & $ O(\log^2 n  )$       \\
			Erd\H{o}s-R\'enyi: $\mathcal{G}(n,(1+\delta)/n)$, $\delta^3n \to \infty$      \cite{din11,din12}         & $ O( (1/ \delta^3)\log^2 (\delta^3n)  )$       \\
			Erd\H{o}s-R\'enyi: $\mathcal{G}(n,1/n)$    \cite{nac08}& $ O(n  )$       \\
			{Newman-Watts (small-world) Graph}   \cite{add15}   & $ O(\log^2 n  )$ \\ 
			Expander Graph  \cite{dur07}    & $ O(\log n )$        \\ 
			%	Exponential Random Graph: High temperature    \cite{bha08}  & $ O( n^2 \log n )$        \\ 
			%	Exponential Random Graph: Low temperature    \cite{bha08}  & $ O(\exp(n) )$        \\ 
			{Any Connected Undirected Graph with Metropolis weights} \cite{Olshevsky2017}  & $ O(n^2 )$        \\ 
			Any Connected Graph    &  $ O(|E| \text{diam}(\mathcal{G}) )$        \\ 
			\hline
		\end{tabular}
		\caption{Upper bounds on the mixing time for various graph topologies.}
		\label{tab:summary_results2}
	\end{table}

\end{document}